\documentclass[12pt]{amsart}

\usepackage{amssymb,latexsym}

\usepackage{graphicx,epsfig}
\usepackage[dvips]{color}

 \makeindex 
 
\def\cal{\mathcal}

\def\Bbb{\mathbb}

\def\rank{\text{\rm rank\,}}

\def\ord{\text{\rm ord\,}}

\def\pad{\phi^a}

\def \supp {\text{\rm supp\,}}

\def\ext{\text{\rm ext}}
\def\tp{{\tilde\phi}}

\def\A{{\cal A}}

\def\M{{\cal M}}
\def\N{{\cal N}}

\def\T{{\cal T}}
\def\bC{{\Bbb C}}

\def\NN{{\Bbb N}}
\def\bN{{\Bbb N}}
\def\bR{{\Bbb R}}
\def\RR{{\Bbb R}}
\def\bZ{{\Bbb Z}}
\def\ZZ{{\Bbb Z}}

\def\vp{{\varphi}}
\def\al{{\alpha}}
\def\be{{\beta}}
\def\ga{{\gamma}}
\def\Ga{{\Gamma}}
\def\la{{\lambda}}
\def\om{{\omega}}

\def\x{(x_1,x_2)}
\def\y{(y_1,y_2)}
\def\pa{{\partial}}

\def\Hess{{\rm Hess}}
\def\ve{{\varepsilon}}
\def\si{{\sigma}}

\def\de{{\delta}}

\def\Om{{\Omega}}
\def\ka{{\kappa}}

\def\pr{\text{\rm pr\,}}

\def\bpm{\begin{pmatrix}}
\def\epm{\end{pmatrix}}

\def\noi{\noindent}
\def\bee{\begin{enumerate}}
\def\ee{\end{enumerate}}

\def\qed{\smallskip\hfill Q.E.D.\medskip}

\newtheorem{thm}{Theorem}[section]

\newtheorem{prop}[thm]{Proposition}
\newtheorem{proposition}[thm]{Proposition}

\newtheorem{lemma}[thm]{Lemma}
\newtheorem{remark}[thm]{Remark}

\newtheorem{assumption}[thm]{Assumption}

\begin{document}


\title[maximal functions  associated to hypersurfaces in $\bR^3$]{Estimates for maximal functions  associated to hypersurfaces in $\bR^3$
with height $h<2:$ Part I}

\author[S. Buschenhenke]{Stefan Buschenhenke}
\address{School of Mathematics, Watson Building,
  University of Birmingham Edgbaston,
Birmingham B15 2TT, United Kingdom}
\email{{\tt buschenhenke@math.uni-kiel.de}}


 \author[S. Dendrinos]{Spyridon  Dendrinos}
\address{School of Mathematical Sciences, University College Cork, 
Western gateway Building, Western Road, Cork, 
Ireland} 
\email{{\tt sd@ucc.ie}}

\author[I. A. Ikromov]{Isroil A. Ikromov}
\address{Department of Mathematics, Samarkand State University,
University Boulevard 15, 140104, Samarkand, Uzbekistan}
 \email{{\tt ikromov1@rambler.ru}}

\author[D. M\"uller]{Detlef M\"uller}
\address{Mathematisches Seminar, C.A.-Universit\"at Kiel,
Ludewig-Meyn-Strasse 4, D-24118 Kiel, Germany} \email{{\tt
mueller@math.uni-kiel.de}}
\urladdr{{http://analysis.math.uni-kiel.de/mueller/}}

\thanks {We acknowledge the support for this work by the Deutsche Forschungsgemeinschaft under  DFG-Grant MU 761/11-1.\\
 This material is based in parts also upon work supported by the National Science Foundation under Grant No. DMS-1440140 while the last  author was in residence at the Mathematical Sciences Research Institute in Berkeley, California, during the Spring 2017 semester.\\
 The first author was supported by the European Research Council grant No. 307617.}

\begin{abstract}
In this article, we continue the study of the problem of $L^p$-boundedness   of the maximal operator $\M$ associated to averages along isotropic dilates of a given,  smooth hypersurface $S$ of finite type  in 3-dimensional Euclidean space.
An essentially  complete answer to  this problem  had been given about seven years  ago   by the last named two authors in joint work with  M. Kempe   \cite{IKM-max} for the case where  the height $h$  of the given surface is at least two.
In the present article,  we  turn to the case  $h<2.$  More precisely, in this Part I, we study the case where $h<2,$  assuming that  $S$ is contained in a sufficiently small  neighborhood of a given point $x^0\in S$ at which  both principal curvatures of $S$ vanish. Under these assumptions and a natural transversality  assumption, we show that, as in the case $h\ge 2,$ the critical Lebesgue exponent for the boundedness of $\M$ remains to be $p_c=h,$ even though the proof of this result turns out to require  new methods, some of which are inspired by the more recent work by  the last named  two authors on Fourier  restriction to $S.$  Results on the case where $h<2$ and exactly one principal curvature of $S$ does not vanish at $x^0$ will appear elsewhere.

 \end{abstract}
\maketitle


\tableofcontents

\thispagestyle{empty}

\setcounter{equation}{0}
\section{Introduction}\label{introduction}

Let $S$ be a smooth hypersurface in $\RR^n$ and let  $\rho\in
C_0^\infty(S)$ be a smooth non-negative function with compact support.  Consider
the associated averaging operators
$A_t, t>0,$ given by 
$$
A_tf(x):=\int_{S} f(x-ty) \rho(y) \,d\si(y),
$$
where $d\si$ denotes the surface measure on $S.$   The 
associated maximal operator is given by 
\begin{equation}\label{1.1}
\M f(x):=\sup_{t>0}|A_tf(x)|, \quad ( x\in \RR^n).
\end{equation}
We remark that by testing $\M$ on the characteristic function of the unit ball in $\RR^n,$ it is easy to see  that a necessary condition for $\M$ to be  bounded  on $L^p(\RR^n)$  is that  $p> n/(n-1),$ provided the transversality assumption \ref{s1.1} below is satisfied.
\smallskip

In 1976,  E.~M.~Stein \cite{stein-sphere} proved that, conversely, if $S$ is the Euclidean unit sphere in $\RR^n,\ n\ge 3, $  then the corresponding spherical maximal operator is bounded  on $L^p(\RR^n)$  for every $p> n/(n-1).$ The analogous result in dimension  $n=2$ was later proven by J.~Bourgain \cite{bourgain85}. The key property of spheres which allows to prove  such  results is the non-vanishing of the Gaussian curvature on spheres.  These results became  the starting point for intensive  studies of various classes of maximal operators associated to subvarieties. Stein's  monography  \cite{stein-book} is an excellent reference to many of these developments.

In the joint work  \cite{IKM-max} of the last-named two authors with M. Kempe,  maximal functions $\M$ associated to smooth hypersurfaces of finite type in $\RR^3$ had been studied under the  following transversality assumption on $S.$ 
\begin{assumption}[Transversality]\label{s1.1}
The affine tangent plane $x+T_xS$ to $S$ through $x$ does not pass through the origin in $\RR^3$  for every $x\in S.$ Equivalently, $x\notin T_xS$ for every $x\in S,$ so that $0\notin S$ and $x$ is transversal to $S$ for every point $x\in S.$ 
\end{assumption}
Let us fix a point $x^0\in S.$  We recall that the transversality  assumption allows us to find a linear change of coordinates in $\RR^3$ so that in the new coordinates
 $S$ can locally be represented as the graph of a function $\phi,$  and that the norm of $\M$ when acting on $L^p(\RR^3)$ is invariant under such a linear change of coordinates.
 More precisely, after applying a suitable  linear change  of coordinates  to $\RR^3$ we may assume that 
$x^0=(0,0,1),$  and that within the neighborhood $U,$ $S$  is given as the  graph
$$
U\cap S=\{(x_1,x_2,1+ \phi\x): \x\in \Om \}
$$
of a smooth function $1+\phi$ defined on an open neighborhood $\Om$ of $0\in\RR^2$ and satisfying the conditions
\begin{equation}\label{1.2}
\phi(0,0)=0,\, \nabla \phi(0,0)=0.
\end{equation}
The measure $\mu=\rho d\si$ is then explicitly given by 
$$
\int f\, d\mu=\int f(x,1+\phi(x)) \eta(x) \,dx,
$$
with a smooth, non-negative  bump function $\eta\in C_0^\infty(\Om),$ and we may write for  $(y,y_3)\in \RR^2\times \RR$
\begin{equation}\label{At}
A_tf(y,y_3)=f*\mu_t(y,y_3)=\int_{\bR^2} f(y-tx, y_3-t(1+\phi(x)))\eta(x) \, dx,
\end{equation}
where $\mu_t$ denotes the norm preserving scaling  of the measure  $\mu$ given by $\int f\, d\mu_t=\int f(tx, t(1+\phi(x)) \eta(x) \,dx.$

Recall also from \cite{IKM-max} that  the {\it height} of $S$ at the point $x^0$  is  defined by 
$h(x^0,S):=h(\phi),$  where $h(\phi)$ is the height of $\phi$ in the sense of Varchenko (which can be computed by means of Newton polyhedra attached to $\phi$). The height  is invariant under affine linear changes of coordinates in the ambient space  $\RR^3.$ 
\medskip

 In  \cite{IKM-max} the authors had given an essentially complete answer to the problem of $L^p$-boundedness of $\M$ when  $h(x^0,S)$ or $p$  are greater or equal  to $2.$ More precisely, if  $h(x^0,S)\ge 2,$  and if the density $\rho$ is supported in a sufficiently small neighborhood  of $x^0,$ then  the condition $p>h(x^0,S)$  is sufficient for $\M$ to be $L^p$-bounded, and this result is  sharp (with the possible exception of the endpoint $p_c=h(x^0,S),$ when $S$ is non-analytic). For an alternative approach to some of these results based on ''damping'' techniques, see also \cite {greenblatt}. Matters change dramatically  when the transversality assumption fails, as has been shown by E. Zimmermann in his doctoral thesis \cite{Zi}.  Zimmermann studied the case where the hypersurface passes through the origin and proved, among other things,  that for analytic $S$ and $\supp \rho$ sufficiently small, the condition $p>2$ is always sufficient for the $L^p$- boundedness of $\M.$ 

\medskip
 In the present article, we return to this problem and look at  the case where $h(x^0,S)<2.$  
It then turns out that there is a big difference in the behaviour of the associated maximal operator, depending on how many of the  principal curvatures of $S$ do vanish at $x^0$ (instances of this phenomenon have already been observed in  articles by  Nagel, Seeger,  Wainger \cite{nagel-seeger-wainger}, and Iosevich and Sawyer \cite{iosevich-sawyer1},  \cite{iosevich-sawyer}, \cite{io-sa-seeger}). The case where both principal curvatures do not vanish at $x^0$ is classical, and here the condition  $p>3/2$ is necessary and sufficient for the $L^p$-boundedness of $\M,$ exactly as in the case of the 2-sphere (see Greenleaf \cite{greenleaf}). Notice that in this case  $h(x^0,S)=1<3/2,$ so that, unlike the case where  $h(x^0,S)\ge 2,$ the height is not the controlling quantity for the maximal operator. Indeed, as mentioned before, the condition   $p>3/2$ is seen to be necessary by testing $\M$ on characteristic functions of small balls, whereas the notion of height is rather related to testing on characteristic functions of the intersection of a ball with a very  thin neighborhood of some hyperplane. 

We shall here  mainly consider the case where both principal curvatures of $S$ do vanish at $x^0,$ i.e., when $D^2\phi(0,0)=0.$ 
In this case, it turns out that the height is still the controlling quantity. More precisely,  our main theorem states the following:

\begin{thm}\label{s1.2}
Assume that $S$ is a smooth, finite-type  hypersurface in $\RR^3$  satisfying  the transversality assumption~\ref{s1.1}, and let $x^0\in S$ be a given  point at which $h(x^0,S)<2$ and both principal curvatures of $S$ do vanish.  

Then there exists a neighborhood $U\subset S $ of the point $x^0$ such
that for every  non-negative density $\rho\in C_0^\infty(U)$ the associated maximal
operator $\M$ is bounded on $L^p(\RR^3)$ whenever $p>h(x^0,S).$

\end{thm}
The condition $p>h(x^0,S)$ is indeed  also necessary, as the following  result shows, which does not require  that both principal curvatures of $S$  vanish at $x^0.$ 

\begin{thm}\label{s1.3}
Assume that  the maximal operator $\M$ is $L^p$-bounded, and that $S$   satisfies  the transversality assumption~\ref{s1.1}.  Then, for every point $x^0\in S$ at which  
$h(x^0,S)<2$ and $\rho(x_0)\neq0,$ we  necessarily have $p>h(x^0,S).$
\end{thm}

Notice  an interesting difference between the cases $h(x^0,S)\ge 2$ and $h(x^0,S)<2:$ in the first case, studied in \cite{IKM-max}, the necessity of the condition $p>h(x^0,S)$ when $\rho(x^0)\ne 0$ could be verified for analytic hypersurfaces, but not for all classes  of smooth, finite type hypersurfaces (where the endpoint $p=h(x^0,S)$ remained open in certain situations). Indeed,  problems with the $L^p$-boundedness of $\M$ at the endpoint $p=h(x^0,S)$ may arise, e.g., for  $\phi$ of the form $\phi\x=x_2^2+f(x_1),$ where $f$ is flat at the origin (we refer to the examples in \cite{IKM-max}, Remark 12.3). Note, however, that in this example $h(\phi)=2,$  and we shall see  that such kind of situation can never arise when $h(x^0,S)<2.$

The case where only one principal curvature of $S$  vanishes, and the other one not (which is the case of singularities  of type $A_n$  in the sense of Arnol'd  - compare Theorem \ref{normalform1}) turns out to be the most difficult one to analyze, and   we shall return to this  problem in subsequent work.  Indeed, it appears that in those  situations where $\phi$ has a  singularity of type $A_n, n\ge 3,$  typically  the optimal necessary conditions for $L^p$-boundedness of $\M$ rather  seem to come from testing on characteristic functions of a very  narrow tubular neighborhood of some line  segment in $\RR^3.$ 
Here, we shall  only look at one instance of singularities of type $A_n,$ namely when $n=2$   (see Theorem \ref{nondegairypar}),  since results on surfaces with $A_2$ - type singularity turn out to be of great relevance also to  the study of $D_4^+$-type singularities, which are included in  Theorem \ref{s1.3}.

\medskip
Before turning to the proof of  Theorem \ref{s1.2}, let us first briefly recall some relevant notation and results from \cite{IKM-max} (compare also the monograph \cite{IMmon}). These sources should also be consulted for further details and references to the topic.  Our analysis in the present  paper will indeed take advantage of several  technics and results developed in those articles.

\setcounter{equation}{0}
\section{Some background on Newton diagrams and adapted coordinates}\label{newton}

\medskip
We first recall   some basic notions from \cite{IM-ada}, which essentially  go back to A.N.~ Var\-chen\-ko \cite{Va}.
Let $\phi$ be a smooth real-valued  function defined on an open  neighborhood $\Omega$ of the origin in $\bR^2$ with $\phi(0,0)=0,\, \nabla \phi(0,0)=0.$ 
Consider the associated Taylor series 
$$\phi(x_1,x_2)\sim\sum_{\al_1,\al_2=0}^\infty c_{\al_1,\al_2} x_1^{\al_1} x_2^{\al_2}$$
of $\phi$ centered at  the origin.
The set
$$\T(\phi):=\{(\al_1,\al_2)\in\bN^2: c_{\al_1,\al_2}=\frac 1{\al_1!\al_2!}\partial_{ 1}^{\al_1}\partial_{ 2}^{\al_2}� \phi(0,0)\ne 0\}      \index{T@$\T(\phi)$ (Taylor support)}
$$
will be called the {\it Taylor support}    of $\phi$ at $(0,0).$  We  shall always assume that the function $\phi$ is of {\it finite type} \index{finite type function} at every point, i.e., that the associated graph $S$ of $\phi$ is of finite type.  Since we are also assuming that $\phi(0,0)=0$ and $\nabla \phi(0,0)=0,$ the finite type assumption at the origin   just means that 
$$\T(\phi)  \ne \emptyset.$$ 
 The
{\it Newton polyhedron}   $\N(\phi)$ \index{N@$\N(\phi)$ (Newton polyhedron)} of $\phi$ at the origin is
defined to be the convex hull of the union of all the quadrants
$(\al_1,\al_2)+\bR^2_+$ in $\bR^2,$ with $(\al_1,\al_2)\in\T(\phi).$  
The associated {\it Newton diagram}  $\N_d(\phi)$ $\N(\phi)$ \index{Nd@$\N_d(\phi)$ (Newton diagram)} in the sense of Varchenko
\cite{Va}  is the union of all compact faces  of the Newton
polyhedron; here, by a {\it face,} \index{face (of a Newton polyhedron)}  we shall  mean an edge or a
vertex.

We shall use coordinates $(t_1,t_2)$ for points in the plane
containing the Newton polyhedron, in order to distinguish this plane
from the $(x_1,x_2)$ - plane.

The {\it Newton distance}  in the sense of
Varchenko, or shorter {\it distance}, $d=d(\phi)$ \index{d@$d(\phi)$ (Newton distance)}
between the Newton polyhedron and the origin  is given by the coordinate $d$ of the point $(d,d)$ at
which the bi-sectrix   $t_1=t_2$ intersects the boundary of the
Newton polyhedron.

The principal face  $\pi(\phi)$ \index{pi@$\pi(\phi)$ (principal face)}  of the Newton polyhedron of
$\phi$  is the face of minimal dimension  containing the point
$(d,d)$. We shall call the
series
$$
\phi_\pr(x_1,x_2):=\sum_{(\al_1,\al_2)\in \pi(\phi)}c_{\al_1,\al_2} x_1^{\al_1} x_2^{\al_2} \index{fipr@$\phi_\pr$ (principal part of $\phi$)}
$$
the {\it principal part}  of $\phi.$ In case that $\pi(\phi)$ is
compact,  $\phi_\pr$ is a mixed homogeneous polynomial; otherwise,
we shall  consider $\phi_\pr$ as a formal power series.

Note that the distance between the Newton polyhedron and the origin
depends on the chosen local coordinate system in which $\phi$ is
expressed.  By a  {\it local  coordinate system (at the origin)} \index{coordinates! local} we 
shall mean a smooth   coordinate system defined near the origin
which preserves $0.$ The {\it height}  of the  smooth function
$\phi$ is defined by
$$
h=h(\phi):=\sup\{d_y\},   \index{height@$h(\phi)$ (height of $\phi$)}
$$
where the
supremum  is taken over all local  coordinate systems $y=(y_1,y_2)$ at the origin, and where $d_y$
is the distance between the Newton polyhedron and the origin in the
coordinates  $y$.

A given   coordinate system $x$ is said to be
 {\it adapted} \index{coordinates!adapted} to $\phi$ if $h(\phi)=d_x.$
In \cite{IM-ada} we proved that one  can always find an adapted
local  coordinate system in two dimensions, thus  generalizing  the
fundamental work by Varchenko  \cite{Va} who worked in the   setting
of  real-analytic functions $\phi$ (see also \cite{PSS}).

Notice  that if the principal face of the Newton polyhedron $\N(\phi)$ is a compact edge, then it lies on a unique {\it principal line} \index{principal! line} 
$$
 L:=\{(t_1,t_2)\in \RR^2:\ka_1t_1+\ka_2 t_2=1\}, 
$$
with $\ka_1,\ka_2>0.$ By permuting the  coordinates $x_1$ and $x_2,$ if necessary, we shall always assume that $\ka_1\le\ka_2.$ The  weight $\ka=(\ka_1,\ka_2)$   \index{weight!kappa@$\ka$} will be called the {\it principal weight} \index{principal! weight} associated to $\phi.$  It induces dilations $\de_r\x:=(r^{\ka_1}x_1,r^{\ka_2} x_2),\ r>0,$ on $\RR^2,$ so that the principal part $\phi_\pr$ of $\phi$ is $\ka$-homogeneous of degree one with respect to these dilations, i.e.,  $\phi_\pr(\de_r\x)=r\phi_\pr\x$ for every $r>0,$ and we find that 
\begin{equation}\nonumber
d=\frac 1{\ka_1+\ka_2}=\frac1{|\ka|}.
\end{equation}
 Notice that $1/|\ka|\le h.$ 
It can then easily be shown (cf. Proposition 2.2 in \cite{IM-ada}) that $\phi_\pr$ can be factorized as 
\begin{equation}\label{2.4}
\phi_\pr\x=c x_1^{\nu_1}x_2^{\nu_2}\prod_{l=1}^M(x_2^q-\la_l
x_1^p)^{n_l},
\end{equation}
with $M\ge 1,$ distinct non-trivial ``roots'' $\la_l\in \bC\setminus\{0\}$ of multiplicities $n_l\in \bN\setminus\{0\},$ and trivial roots of multiplicities $\nu_1,\nu_2\in\bN$ at the coordinate axes.
Here,  $p$ and $q$ are positive integers without  common divisor, and  $\ka_2/\ka_1=p/q.$

\medskip
More generally,  assume  that $\ka=(\ka_1,\ka_2)$ is any weight with $0<\ka_1\le \ka_2$ such that  the line  $L_\ka:=\{(t_1,t_2)\in\bR^2:\ka_1 t_1+\ka_2t_2=1\}$ is a supporting line to the Newton polyhedron $\N(\phi) $ of $\phi$ (recall that a {\it supporting line}  \index{supporting line} to a convex set  $K$ in the plane  is a line such that $K$ is contained in one of the two closed half-planes into which the line divides the plane and such that this line  intersects the boundary of $K$). Then $L_\ka\cap \N(\phi)$ is  a face   of $\N(\phi),$ i.e., either  a compact  edge or a vertex,  and  the   {\it $\ka$-principal part}  of $\phi$
$$
\phi_\ka(x_1,x_2):=\sum_{(\al_1,\al_2)\in L_\ka} c_{\al_1,\al_2} x_1^{\al_1} x_2^{\al_2} \index{fikappa@$\phi_\ka$ ($\ka$-principal part of $\phi$)}
$$
 is a non-trivial polynomial  which is $\ka$-homogeneous of degree $1$ with respect to the dilations associated to this weight as before,  which can be factorized in a similar way as in \eqref{2.4}.  By definition, we then have
\begin{equation}\nonumber 
\phi(x_1,x_2)=\phi_\ka(x_1,x_2) +\ \mbox{terms of higher $\ka$-degree}.
\end{equation}

Adaptedness of a given coordinate system can be verified by means of the following proposition (see  \cite{IM-ada}): 
\smallskip

If $P$ is any  given polynomial which is  $\ka$-homogeneous  of degree one (such as $P=\phi_\pr$), then we  denote  by
\begin{equation}\label{mroots}
n(P):=\ord_{S^1} P   \index{multp@$n(P)$ (maximal order of roots of $P$)}
\end{equation}
the maximal order of vanishing of  $P$   along the unit circle $S^1.$  Observe that by homogeneity, the Taylor support $\T(P)$ of $P$ is contained in the face  $L_\ka\cap \N(P)$ of  $ \N(P).$ We therefore define the {\it homogeneous distance}  of $P$  by 
$d_h(P):= 1/{(\ka_1+\ka_2)}=1/|\ka|.$    \index{dh@$d_h(P)$ (homogeneous distance of $P$)} Notice that $(d_h(P),d_h(P))$ is just the point of intersection of the line $L_\ka$  with the bi-sectrix $t_1=t_2,$ and that $d_h(P)=d(P)$ if and only if $L_\ka\cap \N(P)$ intersects the bi-sectrix.
We remark that the height of $P$ can then easily be  computed by means of the  formula \begin{equation}\label{heightp+}
h(P)=\max\{n(P), d_h(P)\}
\end{equation}
(see Corollary 3.4 in  \cite{IM-ada}). Moreover, in  \cite{IM-ada} (Corollary 4.3 and  Corollary 5.3), we also proved the following characterization of adaptedness of a given coordinate system:
 \begin{proposition}\label{adapted}
The coordinates $x$ are adapted to $\phi$ if and only if one of the following conditions is satisfied:
\medskip

\bee
\item[(a)]  The principal face  $\pi(\phi)$ of the Newton polyhedron  is a compact edge, and $n(\phi_\pr)\le d(\phi).$
\item[(b)] $\pi(\phi)$  is a vertex.
\item[(c)] $\pi(\phi)$ is an unbounded edge.
\ee
\end{proposition}
 
We also note  that in case (a) we have $h(\phi)=h(\phi_\pr)=d_h(\phi_\pr).$  Moreover, it can be shown that we are in case (a)  whenever $\pi(\phi)$ is a compact edge  and $\ka_2/\ka_1\notin\NN;$   in this case we even have  $n(\phi_\pr)< d(\phi)$ (cf. \cite{IM-ada}, Corollary 2.3).

\setcounter{equation}{0}
 \section{Normal forms of $\phi$ under linear coordinate changes when  $h(\phi)<2,$\\
 and proof of Theorem \ref{s1.3}}\label{normalforms}
 
Our approach will be based on the description of  normal forms  of $\phi$ under linear coordinate changes given by the next theorem. The designation of the type of singularity that we list below corresponds to Arnol'd's classification of singularities  (cf. \cite{agv} and  \cite{duistermaat}). However, Arnol'd's normal forms are achieved by means of non-linear coordinate changes, and since we shall be in need of  very precise information on those  coordinate changes, we shall present normal forms which give such precise information (``non-linear shears'' will indeed always lead to adapted coordinates). Our normal forms  will  easily follow by expanding a bit on Proposition 2.11 in \cite{IMmon} in  combination with the proof of Corollary 7.4 in \cite{IKM-max}.

Notice  that in Theorem \ref{s1.2}, we are dealing with Case 2 of the following theorem.

\begin{thm}\label{normalform1}
Assume that $h(\phi)< 2,$ and that $\phi$ has a  degenerate critical point at the origin.   Then, after applying  a suitable linear change of coordinates, $\phi$ can be written  on a sufficiently small neighborhood of the origin in one of the following forms:

\medskip
\noi{\bf  Case 1  (Type A).} $\rank D^2\phi(0,0)=1.$

\begin{equation}\label{AD}
\phi(x_1,x_2)=b(x_1,x_2)(x_2-\psi(x_1))^2 +b_0(x_1),
\end{equation}
where $b,b_0$ and $\psi$ are  smooth functions, and  $b(0,0)\ne 0.$  

Moreover, either $\psi$ is flat at $0,$ or $\psi(x_1)=cx_1^{m}+O(x_1^{m+1})$, with $c\ne0$ and $m\ge 2,$ 
and 

\qquad $b_0(x_1)=x_1^n \beta(x_1),$ with \quad $\beta(0)\ne 0$ \ and \  $n\ge 3.$ \hfill (singularity of  type $A_{n-1})$
\medskip

The coordinates $\x$ are  then adapted to $\phi$ if and only if  $n\le 2m,$ with the understanding that $m:=\infty$ if $\psi$ is flat at the origin. If $n\le 2m,$ then the principal weight is given by $\ka:= ( 1/n,1/2),$ and we have 
$$
d(\phi)=h(\phi)=\frac {2n}{n+2},
$$
and if $n> 2m,$ then the principal weight is given by $\ka:= ( 1/(2m),1/2),$  and Newton distance and height are given by 
$$
d(\phi)=\frac {2m}{m+1}, \qquad h(\phi)=\frac {2n}{n+2}.
$$

\medskip
\noi {\bf Case 2.} $\rank D^2\phi(0,0)=0.$ Here, we distinguish two  subcases.
\smallskip

{\bf (i) Type D.} $\phi$ is still of the  form  \eqref{AD},
with smooth functions   $b,b_0$ and $\psi$ as  before, but  now with $b(0,0)= 0,$  and more precisely
\begin{equation} \label{D}
b\x=x_1 b_1\x+x_2^2 b_2(x_2),
\end{equation}
where $b_1$ and $b_2$ are smooth functions, with $b_1(0,0)\ne 0,$ and 

\smallskip
\qquad  $b_0(x_1)=x_1^n \beta(x_1),$ where $\beta(0)\ne 0$  and $n\ge 3.$ \hfill (singularity of  type $D_{n+1})$ 
 
 \medskip

Here, the coordinates $\x$ are  adapted to $\phi$ if and only if  $n\le 2m+1,$  with the understanding that $m:=\infty$ if $\psi$ is flat at the origin. If $n\le 2m+1,$ then the principal weight is given by $\ka:= ( 1/n,(n-1)/(2n)),$ and  Newton distance and height are given by 
$$
d(\phi)=h(\phi)=\frac {2n}{n+1},
$$
and if  $n> 2m+1,$ then  the principal weight is given by $\ka:= ( 1/(2m+1),m/(2m+1)),$ and  Newton distance and height are given by 
$$
d(\phi)=\frac {2m+1}{m+1}, \qquad h(\phi)=\frac {2n}{n+1}.
$$

\smallskip

{\bf (ii) Type E.}  $\phi$ can be written as 
$$
\phi(x)=\phi_\pr(x) +\phi_r(x),
$$
where the remainder $\phi_r(x)$ comprises all terms of $\ka$-degree strictly bigger than $1,$ and where we may assume that the  principal part $\phi_\pr$ of $\phi$ and the principal weight $\ka$ are from the  following list:

 \begin{itemize}
\item[] $\phi_\pr\x= x_2^3\pm x_1^4,$ \qquad with $\ka:=(1/4,1/3),$ \hfill (singularity of  type $E_6$)
\item[] $\phi_\pr\x= x_2^3+x_1^3x_2,$ \quad with $\ka:=(2/9,1/3),$ \hfill (singularity of  type $E_7$)
\item[] $\phi_\pr\x=x_2^3+x_1^5,$ \qquad with $\ka:=(1/5,1/3).$ \hfill (singularity of  type $E_8$)
 \end{itemize}
 In all these cases, the  coordinates $\x$ are  adapted to $\phi.$ Moreover,
$$d(\phi)=h(\phi)=
 \begin{cases}  \frac {12}7  & \mbox{for type }   E_6, \\
   \frac 95  & \mbox{for type }   E_7, \\  
    \frac {15}8  & \mbox{for type }   E_8. \\ 
\end{cases}
$$
\end{thm}
Notice that in all cases the principal face of the Newton polyhedron of $\phi$ is a compact edge, so that 
$d=d(\phi)=1/|\ka|.$

\begin{remark}\label{rootjet}{\rm 
Following \cite{IMmon}, we shall call the function $\psi$ (respectively its graph) in \eqref{AD} the {\it principal root jet} (compare \cite{IM-ada} for the  general  definition of this notion). Notice that the coordinates $(y_1,y_2):=(x_1,x_2-\psi(x_1))$ are adapted to $\phi$ for singularities of type $A$ or $D,$ and that in these coordinates $\phi$ is of the form
\begin{equation}\label{phinada}
\pad(y_1,y_2)=b^a(y_1,y_2)y_2^2 +b_0(y_1),
\end{equation}
where the function $b^a(y_1,y_2)$ has the same properties as the one described for $b(x_1,x_2).$  }
\end{remark}

\begin{remark}\label{D4}{\rm 
For later purposes, let us observe  that if $\phi$ is of type $D_4,$ then the coordinates $\x$ are already adapted,  and we may assume that the  principal part of $\phi$ is of the form
\begin{equation}\label{D4pm}
\phi_\pr\x= x_1x_2^2\pm x_1^3=x_1(x_2^2\pm x_1^2).
\end{equation}
If $\phi_\pr\x= x_1(x_2^2+ x_1^2),$ then  we say that $\phi$ is of type $D_4^+,$ and if $\phi_\pr\x= x_1(x_2^2- x_1^2),$ then we call $\phi$ of type $D_4^-.$}
\end{remark}

\begin{proof} The statements in Case 1 follow immediately from the proof of Proposition 2.11 in \cite{IMmon}. Notice that if the function $b_0$ were flat at the origin, then we would have $h(\phi)=2.$ Thus singularities of type $A_\infty$ are excluded here. Moreover, if we had $n=2,$ then the origin would be a non-degenerate critical point of $\phi,$  contrary to our assumption.

As for Case 2, following again the proof of Proposition 2.11 in \cite{IMmon}, let us  look at the polynomial $P_3,$ which denotes the homogeneous part of degree $3$ of the Taylor polynomial of $\phi$ with respect to the origin. Let us also denote by $n(P_3)$ the maximal multiplicity of real roots of $P_3.$ 

If $n(P_3)=2,$ then we can just follow the discussion in  \cite{IMmon}  in order to see that $\phi$ is of type $D_{n+1},$ with $n\ge 3$  (note that if $n=2,$ then we would be in  Case 2). As before, the type $D_\infty$ is excluded here, since in that case we would have  $h(\phi)=2.$

If $n(P_3)=1,$ then we have shown in \cite{IMmon} that $P_3$ must be of the form  $P_3\x=x_1(x_2-\alpha x_1)(x_2-\beta x_1),$ where either $\alpha\ne \beta$ are both real, or $\alpha=\overline{\beta}�$ are non-real. It is easy to see that this case can easily be reduced to the form described by \eqref{D4pm} by means of a linear change of coordinates.

Finally, if $n(P_3)=3,$ then as in the proof of Proposition 2.11 in \cite{IMmon} we may assume that $P_3\x= x_2^3.$ But then the proof of  Corollary 7.4 in \cite{IKM-max} shows that $\phi$ is of type $E_6, E_7$ or $E_8.$ 

The remaining statements are easily  verified using Proposition \ref{adapted}.
\end{proof} 

\medskip
Given these normal forms for $\phi,$ it is now easy to prove the necessary condition for $L^p$-boundedness of the maximal operator $\M$ in Theorem \ref{s1.3}. Indeed,  from our normal forms one deduces  that the principal face of the Newton polyhedron of $\phi,$ when expressed as $\pad$ in adapted coordinates, is a compact edge (compare \eqref{phinada} in Remark \ref{rootjet} for singularities of type $A$ and $D$). Then Theorem \ref{s1.3} is an immediate consequence of Proposition 12.1  in \cite{IKM-max} in combination with a result by  Iosevich and Sawyer \cite{iosevich-sawyer1}, namely  Theorem 1.5 in \cite{IKM-max}.

\medskip

In this paper, we shall mainly study maximal functions associated with   Case  2  in Theorem \ref{normalform1} where $D^2\phi(0,0)=0.$  However, as we shall see, the study of $D_4^+$ - type singularity will also require the understanding of maximal functions associated to surfaces with $A_2$ - type singularities.

\medskip 
In the next section, we shall provide an auxiliary estimate for maximal functions associated to families of hypersurfaces depending and some perturbation parameter $\si$ and some large  translation parameter $T$ (translations of the hypersurfaces in transversal directions).  The corresponding estimates will become useful in many situations.

\setcounter{equation}{0}
\section{Auxiliary estimates for maximal operators}\label{maxpar}

\subsection{An estimate for maximal operators depending on parameters}

Let us consider a smooth family $S^{\si,T}$ of hypersurfaces in $\RR^{n+1}$ given as graphs
\begin{equation}\label{firstfunc}
S^{\si,T}:=\{(x,T+\phi(x,\si)): x\in U\},
\end{equation}
where $\phi=\phi(x,\si)$ is  a smooth, real-valued function defined on an open neighborhood   $U\times V$ of a  given point $(x_0,\si_0)$ in $\RR^n\times \RR^m,$ and $T$ is a large real  translation parameter.
Denote  furthermore by
$\mu^{\si,T}$ the surface-carried measure on $S^{\si,T}$ defined by 
$$
\int f\, d\mu^{\si,T}:=\int f(x,T+\phi(x,\si))\, a(x,\si) dx,
$$
where $a(x,\sigma)$ is a nonnegative smooth function with compact support in $U\times V.$
For $t>0$ we denote by $\mu_t^{\si,T}$ the measure-preserving scaling of $\mu^{\si,T}$ given by 
$$
\int f\, d\mu_t^{\si,T}=\int f\big(tx,t(T+\phi(x,\si))\big)\, a(x,\si) dx,
$$ 
and consider the averaging operator
\begin{equation}\label{average}
A^{\si,T}_t f(y, y_{n+1}):=f*\mu_t^{\si,T}(y,y_{n+1})=
\int_{\bR^n} f\big(y-tx,y_{n+1}-t(T +\phi(x,\sigma))\big)\, a(x,\sigma)dx.
\end{equation}
 Define the
associated maximal operator by 
 \begin{equation}\label{max}
\M^{\si,T}f(y,y_{n+1}):=\sup_{t>0}| A^{\si,T}_t f(y,y_{n+1})|, \qquad (y,y_{n+1})\in \RR^{n+1}.
\end{equation}

\begin{prop}\label{simple}
Assume that the  uniform estimate
\begin{equation}\label{fdecay1}
\left|\int e^{i(\xi \cdot x+\xi_{n+1}\phi(x,\sigma))}\eta(x)dx\right|\le C
\frac{\|\eta\|_{C^{n+1}}}{(1+|\xi_{n+1}|)^{\gamma}}, \qquad (\xi,\xi_{n+1})\in\RR^{n+1},
\end{equation}
holds true  for every $C^{n+1}$- function $\eta$ with compact support in $U,$ where $C$ is independent of $\eta$ and 
$\si,$
and that  $n/2\ge \gamma>1/2.$ Then for every  $p>1+1/(2\gamma)$  and $|T|\ge 1$ the maximal
operator $ \M^{\si,T}$ is  bounded on $L^p(\RR^{n+1}),$ with norm 
 $$
 \|\M^{\si,T}\|_{p\to p}\le C_p|T|^\frac1{p},
 $$
 where the constant $C_p$ is independent of $\si$ and $T.$
\end{prop}

\begin{proof} Since for the proof we can essentially follow a by now well-known pattern (see, e.g., \cite{stein-book}, Ch. XI. 3), we shall only sketch the argument.  Let us assume without loss of generality that $(x_0,\si_0)=(0,0).$ Moreover, in order to facilitate the notation, we shall drop superscripts $\si,T$ and write  $\mu$ for $\mu^{\si,T},$ etc..

We choose  smooth non-negative bump functions $\chi_0$ supported in $[-1,1]$ and $\chi_1$ supported in $[-2,-1/2]\cup [1/2,2]$ such that
$$
\sum_{l=0}^\infty \chi_l(s)=1 \quad \mbox{for all }\quad s\in \RR,
$$
where $\chi_l(x):=\chi_1(2^{1-l}x)$ for $l\ge 1.$ Then we perform the inhomogeneous Littlewood-Paley decomposition 
$$
\mu=\sum_{l=0}^\infty \mu^l, \qquad \mbox{where} \qquad\widehat{\mu^l}(\xi,\xi_{n+1}):=\chi_l(\xi_{n+1})\hat \mu(\xi,\xi_{n+1}),
$$
and denote by $\M^l$ the maximal operator associated to $\mu^l$ in place of $\mu,$ i.e., $\M^l f=\sup_{t>0} |A_t^lf|,$ with $A^l_t f(y, y_{n+1}):=f*\mu^l_t(y,y_{n+1}).$

We first estimate $\M^l$ on $L^2.$ 
By writing $t=t'2^{-j}$ with $t'\in[1,2[$ and $j\in \ZZ,$ it is easily seen that 
$$
\M^lf(y,y_{n+1})\le \Big(\sum_{j\in\ZZ} \sup_{t\in [1,2[}|A^l_{t2^{-j}} f(y, y_{n+1})|^2\Big)^{1/2},
$$
and since $\widehat{\mu^l}(t2^{-j}\,\cdot)$ is supported in the set where $|\xi_{n+1}|\sim 2^{l+j},$ we may replace $f$ on the right-hand side of this inequality by $\sum_{k=-2}^2\Delta _{l+j+k}f,$ with $\widehat{\Delta_mf}(\xi,\xi_{n+1}):=\chi_m(\xi_{n+1})\hat f(\xi,\xi_{n+1}).$ Since the functions $\Delta_mf$ have almost disjoint Fourier supports, we easily see by means of Plancherel's theorem that it suffices to prove an  estimate for the local maximal functions $\M^l_j f=\sup_{1\le t<2} |A^l_{t2^{-j}}f|,$  of the form
\begin{equation}\label{maxloc}
\|\M^l_j f\|_2\le C_l\|f\|_2,\qquad \mbox{with }\ C_l \ \mbox{independent of} \  j,
\end{equation}
 in order to derive a corresponding estimate $\|\M^l f\|_2\le C C_l\|f\|_2.$ But, notice that for $l\ge 1$
 $$
 \widehat{\mu^l}(t2^{-j}(\xi,\xi_{n+1}))=\chi_1(t2^{-j-l}\xi_{n+1})\int e^{it2^{-j}(\xi\cdot x+\xi_{n+1}(T+\phi(x,\si))} a(x,\si) \, dx.
 $$
Our assumption \eqref{fdecay1} thus easily implies that, for $1\le t<2,$
\begin{eqnarray*}
 |\widehat{\mu^l}(t2^{-j}(\xi,\xi_{n+1}))|\le C 2^{-l\ga},\\
  |\pa_t\widehat{\mu^l}(t2^{-j}(\xi,\xi_{n+1}))|\le C 2^{-l\ga}2^l|T|
\end{eqnarray*}
(observe here that if $|\xi|/|\xi_{n+1}|$ is sufficiently large, then iterated integrations by parts lead to  the stronger estimate $\left|\int e^{i(\xi \cdot x+\xi_{n+1}\phi(x,\sigma))}\psi(x)dx\right|\le C\frac{\|\psi\|_{C^n}}{(1+|\xi|)^n}.$)
By means of  a variant of the Sobolev embedding theorem (compare  \cite{stein-book}, Ch. XI. 3.2) and Plancherel's theorem we then find that the norm of $\M^l_j$ can be estimated by the geometric mean of the two right-hand sides of these estimates, so that we may choose $C_l=C 2^{-l\ga}2^{l/2}|T|^{1/2}$ in \eqref{maxloc}. Consequently,
\begin{equation}\label{Mk2}
\|\M^l\|_{2\to 2}\le C 2^{l/2}2^{-l\ga}|T|^{1/2}.
\end{equation}

As for the estimation of $\M_l$ on $L^1,$ observe that, except for some Schwartz tail,  $\mu^l$ is essentially supported in a  cuboid of dimensions comparable to 
$1\times \cdots\times 1\times |T|,$ and that $\|\mu^l\|_\infty \le C 2^l.$ This allows to dominate $\M^lf$ by $C2^l |T| M_{HL}(|f|),$ where $M_{HL}$ denotes the Hardy-Littlewood maximal operator.  Therefore we have the weak-type estimate
\begin{equation}\label{Mk1}
\|\M^l\|_{L^1\to L^{1,\infty}}\le C 2^{l}|T|.
\end{equation}
From these two  estimates \eqref{Mk2} and \eqref{Mk1} we obtain  by means of  Marcinkiewicz' interpolation theorem (see, e.g., \cite{grafakos}) that 
$$
\|\M^l\|_{L^p\to L^p}\le C_p 2^{l(\frac 1p -2\ga(1-\frac 1p))}|T|^{1/p}.
$$
Very similar arguments apply when $l=0.$ 
Thus, if $p>1+1/(2\ga),$  these estimates sum in $l$  and we arrive at the desired estimate $\|\M\|_{L^p\to L^p}\le C_p |T|^{1/p}$ for this range of $p$'s.
\end{proof}

\subsection{A variation on Hardy-Littlewood's maximal operator}
The following result, which may also be of independent interest,  will become relevant to our ''Airy-type'' analysis in Section \ref{ndairy}.

If $A$ is a bounded Lebesgue measurable subset of  $\RR^n,$  then  we denote by $\M_A$ the corresponding maximal operator 
$$
\M_A f(x):=\sup_{t>0} \int_A |f(x+ty)|\, dy, \qquad f\in L^1_{\rm loc}(\RR^n).
$$
In particular, if $A=B_1(0)$ is the Euclidean unit ball, then $\M_A$ is the Hardy-Littlewood maximal operator $\M_{\rm HL}.$

Denote further by $\pi:\RR^n\setminus\{0\}\to S^{n-1}$ the {\it spherical projection}  onto the unit sphere $S^{n-1},$  given by $\pi (x):= x/|x|,$     and by $|\pi(A)|$ the $n-1$-dimensional volume of this set with respect to the surface measure  on the sphere.

\begin{proposition}\label{maxproj}
Assume that  $A $ is an open subset of $\RR^n$   contained in the annulus $R\le |x|<2R,$ where $R>0.$  
Then, for $1<p\le \infty,$   we have that 
\begin{equation}\label{maxA1}
\|\M_A f\|_{L^p\to L^p}\le C_p R^n |\pi(A)|.
\end{equation}
Moreover, if $p=1,$ then we have that  $\|\M_A f\|_{L^1\to L^{1,\infty}}\ge c_1 R^n|\pi(A)|.$
Here, $c_1>0$ and $C_p$ are positive constants  which are independent of the set $A.$ 
\end{proposition}
\noi {\bf Remarks:}  The  estimate \eqref{maxA1} gets only sharp as $p\to 1.$ For example, if $A=S_\de, \de>0,$ is the $\de$-neighborhood of the sphere $S^{n-1},$ then the  boundedness of the spherical maximal operator on $L^p$ for $p>n/(n-1)$ implies that in this range of $p$'s, 
$\|\M_{S_\de} f\|_{L^p\to L^p}\le C \de,$ whereas $\pi(S_\de)=S^{n-1}.$

We do not know if for $p=1$  a weak-type estimate of the form $\|\M_A f\|_{L^1\to L^{1,\infty}}\le C_1 |\pi(A)|$ holds true. The  argument that we shall use in the proof  does not  allow to show this, since weak-type estimates only sum with a kind of logarithmic loss  (cf. \cite{SW69})

\begin{proof} Scaling by $1/R,$ we can easily reduce considerations to the case $R=1.$ 
By an  {\it $\ve$-tube} ($\ve>0)$ we shall mean in this proof  any tube  of length 6 and radius $\ve$ centered at the origin.  By a standard Whitney  decomposition of the open subset $\pi(A)$ of the sphere, we may find a sequence  $T_j$ of such tubes, where $T_j$ is an $\ve_j$-tube, such that  the $T_j$ cover the set $\pi(A)$ with bounded overlap. We can do this even in such a way that the $T_j$ also cover the set $A.$  Then 
$$\M_Af(x)\le \sum_j\M_{T_j} f(x).
$$
Moreover, a simple scaling argument allows to compare  the operators 
$\M_{T_j}$  with the  Hardy-Littlewood maximal operator, and we find that 
$\|\M_{T_j} f\|_{L^p\to L^p}\le C'_p |T_j|=C''_p\ve_j^{n-1}\le C'''_p |T_j\cap S^{n-1}|.$ 
This implies that $\|\M_{A} f\|_{L^p\to L^p}\le C_p |\pi(A)|.$ 
\medskip

To prove an inverse estimate for $p=1,$   let $\ve>0,$ and put $f:=\chi_{B_{\ve}(0)}.$ Denote  by $A_\ve$ the set of all points in $A$ whose $\ve$-neighborhood is also contained in $A.$ Then, for every point $a\in A_\ve,$ we see that $\M_A f(x)\ge C \ve^n$ on $1/8$'th of  the $\de$-tube 
passing through $-a.$ Therefore $\M_A f(x)\ge C \ve^n=C'\|f\|_1$ on a set of measure $\ge c|\pi(A_\ve)|,$ where $c>0$ is a fixed constant. This implies that  $\|\M_A f\|_{L^1\to L^{1,\infty}}\ge c_1 |\pi(A_\ve)|,$ for every $\ve>0.$ The asserted inverse estimate for $p=1$ follows, since $|\pi(A_\ve)|\to |\pi(A)|$ as $\ve\to 0.$
\end{proof}

\setcounter{equation}{0}
\section{Estimation of  the maximal operator $\M$ in the  presence of a linear coordinate system which is adapted to $\phi$}
\label{adaptedc}
We begin the proof of Theorem \ref{s1.2} with the discussion of the cases where the  coordinates $\x$  in Theorem \ref{normalform1} are adapted to $\phi,$  which strongly facilitates the arguments.  Recall that in these cases 
$$
d=\frac 1{|\ka|}=h.
$$

Following  the approach in Section 7 of \cite{IKM-max}, given the principal weight $\ka,$ we first perform a dyadic decomposition  by means of the  dilations 
$\de_{r}\x:=(r^{\ka_1}x_1,r^{\ka_2}x_2),$ $r>0,$  associated to $\ka.$ To this end, we choose  a
smooth non-negative function 
$\chi_1$ supported in the annulus $\A:=\{1\le|x|\le R\}$ satisfying
$$
\sum_{k=k_0}^\infty\chi_1(\de_{2^k}x)=1 \quad \mbox{for}\quad 0\ne x\in \Om.
$$
 Notice that by choosing $\Om$ small, we can choose $k_0\in\NN$ as large as we need.  
We then decompose the measure $\mu=\rho d\si,$ which is explicitly given by 
$$
\int f\, d\mu=\int f(x,1+\phi(x)) \eta(x) \,dx,
$$
with a smooth, non-negative  bump function $\eta\in C_0^\infty(\Om),$ accordingly as 
$$
\mu=\sum_{k=k_0}^\infty\tilde\mu_k,
$$
with 
$$
\int f\, d\tilde\mu_k=\int f(x,1+\phi(x))\, \eta(x)\chi_1(\de_{2^k}x) dx.
$$
It will then suffice to derive suitable $L^p$-estimates for the maximal operators $\sup_{t>0}|f*(\tilde\mu_k)_t|.$ 
Applying a straight-forward $L^p$-isometric re-scaling to them by means of the dilations $\de_{2^{-k}},$ we may assume that these are of the  form $2^{-|\ka| k}\M_k f,$ where 
$$
\M_kf(y,y_3):=\sup_{t>0}|f*(\mu_k)_t(y,y_3)|
$$
and 
$$
\int f\, d\mu_k:=\int f(x,2^k+\phi^k(x))\, \eta(\delta_{2^{-k}}x)\chi_1(x) dx.
$$
Here we have set 
$$\phi^k(x):=2^k\phi(\de_{2^{-k}}(x))=\phi_\pr(x)+2^k\phi_r(\de_{2^{-k}}(x)).$$ 
Notice that the perturbation  term $2^k\phi_r(\de_{2^{-k}}(\cdot))$ is of order
$O(2^{-\ve k})$  for some $\ve>0$ in any $C^M$-norm. To express this fact, we shall in the sequel again  use the short-hand notation
$2^k\phi_r(\de_{2^{-k}}(\cdot))=O(2^{-\ve k}).
$
To summarize, we shall then have
\begin{equation}\label{mtomk}
\|\M\|_{L^p\to L^p}\le \sum_{k= k_0}^\infty  2^{-|\ka| k}\|\M_k\|_{L^p\to L^p}.
\end{equation}

For the  estimation of $\M_k$  we shall invoke Proposition \ref{simple}, with $T:=2^k$ and $\si:=2^{-k}.$ We then have to estimate  oscillatory integrals of the form
$$
J(\xi,\xi_3):= \int e^{i(\xi \cdot x+\xi_{3}\phi^k(x))}\eta(x)dx,
$$
where $\eta$ is smooth with compact support in the annulus $\A$ on which $|x|\sim 1.$ As in the proof of Proposition \ref{simple} we may and shall assume that $|\xi|\lesssim |\xi_3|.$ This allows us to re-write 
$$
J(\xi,\xi_3)= \int e^{i|\xi_3|\phi(x,s,\si)}\eta(x)dx,
$$
where the  complete phase in this oscillatory integral is of the form
$$
\phi(x,s,\si)=\phi(x,\si)+s\cdot x,
$$
with   $s=\xi/|\xi_3|\in \RR^2$ satisfying $|s|\lesssim 1,$  and where $\phi(x,0)=\phi_\pr(x).$ 
\medskip

Now {\it assume} that we can  estimate $J(\xi,\xi_3)$ by  
\begin{equation}\label{je}
|J(\xi,\xi_3)|\lesssim \frac{\|\eta\|_{C^3}}{(1+|\xi_3|)^\ga}
\end{equation}
uniformly in $\si,$ with some $\ga$ such that
\begin{equation}\label{hga}
d\ge 1+1/{(2\ga)}.
\end{equation}
 Then Proposition \ref{simple}  implies that  
\begin{equation}\label{jest1}
2^{-|\ka|k}\|\M_k\|_{L^p\to L^p}\lesssim
2^{-|\ka|k}(2^k)^{\frac{1}{p}}=2^{-(\frac{1}{d}-\frac{1}{p})k}, 
\end{equation}
provided $p\ge d,$ and if $p>d,$ then the series in \eqref{mtomk} converges and we see that $\|\M\|_{L^p\to L^p}<\infty$  whenever $p>d.$  Thus, whenever we can verify
 \eqref{je} and \eqref{hga}, then we obtain the desired estimate in Theorem \ref{s1.2}.
\medskip

We begin with singularities of type $E.$

\medskip
  {\bf The case $E_6$}. Here $\phi_\pr\x= x_2^3\pm x_1^4$ and  $\ka:=(1/4,1/3),$ and we claim that estimate \eqref{je} holds with $\ga=3/4.$

To this end, fix $x^0$ in the support of $\eta,$ $s^0$ and $\si^0:=0.$  We want to estimate  the contribution of a small neighborhood  of $x^0$ to the integral $J(\xi,\xi_3),$  uniformly for all $(s,\si)$ in a small neighborhood of 
$(s^0,0).$ If the complete phase $\phi(x,s^0,\si^0)$ has no critical point at $x^0,$ then integrations by parts lead to even stronger estimates than required by \eqref{je}. Let us therefore assume that $x^0$ is a critical point.
Since $|x^0|\sim 1,$ at least one of the two coordinates of $x^0$ is of size $1.$ 
This shows that in order to estimate  the contribution of a small neighborhood  of $x^0$ to the integral $J(\xi,\xi_3),$ we may first apply the method of stationary phase in one of the two variables of integration, and subsequently van der Corput's lemma of order either $2,3$ or $4$ in the remaining variable,  leading to an estimate of order $O(|\xi_3|^{-1/2-1/4})$  for all $(s,\si)$ sufficiently close to $(s^0,\si^0),$ in the worst case scenario. By means of a partition of unity argument, this leads to \eqref{je}, with $\ga=3/4.$ 
Since here $d=12/7>5/3=1+1/(2\ga),$ we are done.
\medskip

  {\bf The case $E_8$}. Here, a very similar argument applies and we obtain the estimate \eqref{je} with $\ga=7/10.$ 
Since $d=15/8>12/7=1+1/(2\ga),$ we are again done.
\medskip

{\bf The case $E_7$}. In this case we have
$\phi_\pr=x_2^3+x_1^3x_2,$ and   we claim that here the estimate \eqref{je} holds with $\ga=5/6.$

Indeed,  $\pa_2^2\phi_\pr(x^0)= 6x^0_2$ and $\pa_1^2\phi_\pr(x^0)= 6x^0_1 x^0_2.$
Thus, if $x^0_2\ne 0,$  and if we  assume again that $x^0$ is a critical point of the complete phase,  then near $x^0$ we can again first apply the method of stationary phase in $x_2$  and then van der Corput's lemma  of order $3$ in $ x_1$  and arrive at the desired  estimate. And, if $x^0_2=0,$ then $x_1^0\ne 0,$ and one checks that the Hessian determinant of $\phi_\pr$ does not vanish at $x^0.$ Thus we can apply the method of  stationary phase to a small neighborhood of $x^0$ and  find that for this we may even choose $\ga=1$ in  \eqref{je}.  So, in all cases, \eqref{je} holds true with $\ga =5/6.$ 

 Since $d=9/5>8/5=1+1/(2\ga),$ we are again done.

\medskip

  {\bf The case $D_{n+1} (n\ge 3).$ }
We next turn to singularities of type $D_{n+1},$  assuming that the coordinates $\x$ are adapted to $\phi,$ i.e., that $3\le n\le 2m+1.$
The case where $n=2m+1$ is a bit more delicate, so let us first look at 
\smallskip

{\bf a) The case  $n\le 2m.$}
Assuming without loss of generality that $b_1(0,0)=1,$ then under the condition $n\le 2m$ the principal part of $\phi$ is given by
$$
\phi_\pr\x=x_1 x_2^2+\beta(0)x_1^n, \quad \mbox{with}\ \beta(0)\ne 0.
$$
We  may then argue in a very similar way as in case $E_7$ to see that  here $J(\xi,\xi_3)$ can again be estimated by \eqref{je} with $\ga=5/6.$ Since   for $n\ge 4$ we have $d=2n/(n+1)\ge 8/5,$  we are  done in the case $n\ge 4.$

\smallskip

{\bf b) The case  $n= 2m+1.$}  We claim that in this case $J(\xi,\xi_3)$ can  be estimated by \eqref{je} with $\ga=3/4.$ 

In order to prove this,  notice that we may assume without loss of generality that the principal part of $\phi$ is of the form
$$
\phi_\pr\x=x_1(x_2-cx_1^m)^2+\beta(0)x_1^{2m+1},  \quad \mbox{with}\   c, \beta(0)\ne 0.
$$
In order to estimate $J(\xi,\xi_3),$ let us fix as before $x^0, s^0$ and $\si^0:=0$ with $|x^0|\sim1$ and $|s^0|\lesssim 1.$  Assume also that $x^0$ is the critical point of the phase function $\phi(x,\,s^0,\,0)$. Then necessarily $|s^0|\sim1.$  Let us pass to new adapted coordinates $(y_1,y_2):= (x_1, x_2-cx_1^m)$  in the integral defining $J(\xi,\xi_3).$ In these coordinates, the complete phase for $\si=0$ is then given by $ \pad(y_1,y_2,s^0,0):=\phi(y_1,y_2+cy_1^m, s^0,0),$ i.e., by 

 $$
 \pad(y_1,y_2,s^0,0)=y_1y_2^2+\beta(0)y_1^{2m+1}+s_1^0y_1+s_2^0cy_1^m+s_2^0y_2.
 $$
$y^0$ will denote the  critical point of this phase corresponding to $x^0.$ Obviously,  we have $|y^0|\sim1$. We distinguish now two cases:

\smallskip

{\bf Case 1. $y_1^0=0.$ }Then  necessarily $y_2^0\neq0$ and $0=\pa_2\pad(y^0, s_2^0,\,0)=s_2^0.$
In this case   $\Hess (\pad)(y^0,s^0,0)=-4(y_2^0)^2\neq0$.  So, $y^0$ is a non-degenerate critical point and we can use  the method of stationary phase  in two variables to obtain  an estimate of order $O(|\xi_3|^{-1})$ for $J(\xi,\xi_3),$  which is  stronger than what we need.

\smallskip
{\bf Case 2. $y_1^0\neq0$.} Then, for given $y_1\ne 0 $ in a sufficiently small neighborhood of the point $y_1^0,$ the critical point of the phase with respect to the variable $y_2$  is given by $y_2^c(y_1):=-s_2^0/(2y_1).$   Clearly it is non-degenerate. Applying the method of  stationary phase method in  $y_2$ we  thus arrive at the new phase 
$$
\phi_1(y_1):= \pad(y_1,y_2^c(y_1),s^0,0)=\be(0)y_1^{2m+1}+s_1^0y_1+s_2^0cy_1^m-\frac{(s_2^0)^2}{4y_1}.
$$
Since the exponents of $y_1$ in this phase  are all different, namely $2m+1,1, m$ and $-1$  the equation
$$
\phi_1'(y_1)=0
$$
 can have at most a   root of multiplicity $3$ at the point $y_1^0,$ and thus van der Corput's estimate (more precisely its variant given by Lemma 2.1 in \cite{IMmon}) implies that the remaining one-dimensional  integral in $y_1$ admits a uniform estimation of order  $O(|\xi_3|^{-1/4}).$  In conclusion, by first applying the method of stationary phase in $y_2,$  and then this van der Corput type estimate in $y_1,$ we  see that  estimate \eqref{je} holds true, with $\ga:=1/2+1/4=3/4.$ 
  Since   for $n\ge 4$ we have $d=\frac{2m+1}{m+1} \ge5/3= 1+1/(2\ga),$  we are  done also  in this case.

\medskip 
There remains the case $n=3,$ i.e., the case of $D_4$ - type singularities. Notice that in this case the coordinates $\x$ are always adapted, since $n=3\le 5\le 2m+1.$   We may assume that 
$$
\phi_\pr\x=x_1 x_2^2\pm x_1^3, 
$$
i.e., that we have a $D_4^+$ or a $D_4^-$ - type singularity (compare Remark \ref{D4}). The case of a $D_4^-$- type singularity is easy, since in this case the Hessian determinant of $\phi_\pr$ is given by 
$\Hess(\phi_\pr)(x_1,x_2)=-12x_1^2-4x_2^2\neq0$ whenever $x\neq0$. Thus, on the annulus $D,$ there are only non-degenerate critical points of $\phi,$ and therefore the method of  stationary phase implies that estimate \eqref{je} holds with $\ga=1.$ Since  here $d=3/2=1+1/(2\ga),$we are again done.

 \smallskip
 
There remains the  $D_4^+$ case, in which the Hessian determinant 
$\Hess(\phi_\pr)(x_1,x_2)=12x_1^2-4x_2^2$  may vanish on the annulus.  This case will require a more refined analysis, which will be carried out in the last two sections Sections \ref{ndairy}, \ref{D4est}.

\begin{remark}\label{problemA}
{\rm If $\phi$ has a singularity of type $A_{n-1},$ with $3\le n\le 2m,$ so that the coordinates $\x$ are adapted to $\phi,$ then  a similar argument would show that estimate \eqref{je} holds with $\ga=1/2+1/n=1/d=1/h.$  But then the condition \eqref{hga} would  just mean that $h\ge 2,$   in contradiction to our assumption $h<2.$  This gives a first hint that the treatment of type $A$ singularities will require much finer methods.}
\end{remark}

\setcounter{equation}{0}
\section{Estimation  of  the maximal operator $\M$  when there is no linear adapted coordinate system}\label{nonadac}
Assuming that the coordinates $\x$ are not adapted to $\phi$ in Theorem \ref{normalform1} means that we  are dealing with singularities of type $D_{n+1}, n\ge 3,$  so that $\phi$ is of the form \eqref{AD}, with $b\x$ given by \eqref{D}, i.e.,
\begin{equation}\label{nonadapted}
\phi(x):=x_1(b_1(x_1,x_2)+x_2^2b_2(x_2))(x_2-x_1^m\omega(x_1))^2+x_1^n\beta(x_1),
\end{equation}
where $b_1,b_2,\beta$ are smooth functions with
$b_1(0,0)\neq0$, $\omega(0)\neq0$ and $\beta(0)\neq0$. 
Note also that $n\ge 2m+2$ and  $m\ge2,$  so that in particular $n\ge 6$ and hence  
\begin{equation}\label{hest}
h=\frac{2n}{n+1}\ge \frac{12}7.
\end{equation}

As in \cite{IKM-max}, the main problems  will  here arise from a sufficiently narrow 
neighborhood of the principal root jet, i.e.,  the curve $x_2=\psi(x_1)=x_1^m\om(x_1).$
More precisely, consider the function $\pad(y_1,y_2)$ in \eqref{phinada} of  Remark \ref{rootjet}, which describes $\phi$ in the adapted  coordinates $(y_1,y_2)=(x_1, x_2-\psi(x_1)).$ Then one easily sees that 
$$
\pad(y)=y_1(b^a_1(y_1,y_2)+y_2^2b^a_2(y_2))y_2^2+y_1^n\beta(y_1),
$$
where $b^a_1(0,0)=b_1(0,0)\ne 0\ne \be(0),$ and therefore  
the principal part of 
$\pad$ is given by 
\begin{equation}\label{padpr}
\pad_\pr\y=y_1b_1(0,0)y_2^2+y_1^n\be(0),
\end{equation}
which is $\ka^a$-homogeneous of degree one with respect to the weight 
$$
\ka^a:=\Big(\frac 1n,\frac{n-1}{2n}\Big).
$$
Notice also that 
$$
h(\phi)=\frac 1{|\ka^a|} \quad \mbox{and  }  a:=\ka^a_2/\ka^a_1=(n-1)/2>m.
$$
Following our approach from  \cite{IKM-max}, we therefore    decompose the domain $\Om$ into three regions, an ``exterior'' region of the form
\begin{equation}\label{Dext}
D_\ext :=\{x\in \Om: 
|x_2-\psi(x_1)|\ge  \ve |x_1^{m}|\},
\end{equation}
the ``principal region'' close to the principal root jet,  which is of the form
\begin{equation}\label{Dpr}
D_\pr :=\{x\in \Om: 
|x_2-\psi(x_1)|\le  N|x_1^{a}|\},
\end{equation}
and a ``transition region'' of the form
\begin{equation}\label{transr}
E :=\{x\in \Om: 
N|x_1^{a}|\le|x_2-\psi(x_1)|\le \ve|x_1^{m}|\},
\end{equation}
where $\ve>0$ is a sufficiently small and $N>0$ a sufficiently large parameter to be chosen later.  

Observe also  that the region $D_\ext$ is  essentially  invariant under the dilations $\de_r, r\ll 1,$ associated to the weight 
$\ka,$ whereas the region $D_\pr,$  when expressed in the adapted coordinates $\y,$ is  essentially  invariant under the dilations $\de^a_r, r\ll1,$ associated to the weight $\ka^a,$  i.e., the dilations defined by 
$\de^a_r\y:=(r^{\ka^a_1}y_1,r^{\ka^a_2}y_2).$
 \smallskip
 
For the localizations to such regions, here  and in the sequel, it will be useful to employ the following notation   from \cite{IKM-max}: if  $A_t$  is the averaging operator  from \eqref{At}, and if $\chi$ is any integrable ``cut-off'' function defined on $\Om,$ we shall denote by $A^\chi_t$ the correspondingly localized averaging operator 
$$
A_t^\chi f(y,y_3)=f*\mu^\chi_t(y,y_3)=\int_{\bR^2} f(y-tx, y_3-t(1+\phi(x)))\eta(x)\chi(x) \, dx
$$
corresponding to the measure $\mu^\chi:=(\chi\otimes 1)\mu,$ 
and by $\M^\chi$ the associated maximal operator
$$\M^\chi f(y,y_3):=\sup_{t>0}|A^\chi_tf(y,y_3)|.
$$

\subsection{The contribution by  the region away from the  principal root jet}\label{nearjetsec}
Here we may essentially proceed as in Section \ref{adaptedc}, choosing again for $\ka$ the principal weight $\ka:= ( 1/(2m+1),m/(2m+1)).$

In order to localize to  an  exterior region $D_\ext$ in a smooth fashion, we  fix a cut-off function $\rho\in C_0^\infty (\RR)$ supported in $[-2,2]$ and  identically $1$ on  $[-1,1],$  and put 
$$
\rho_1\x:=\rho\Big(\frac{ x_2-\om(0)x_1^m}{\ve x_1^{m}}\Big), \qquad \ve_1>0.
$$
Since $\psi(x_1)=\om(0)x_1^m+o(|x_1|^m),$  we could as well have chosen $\psi(x_1)$ in place of the leading term 
$\om(0)x_1^m$ of $\psi(x_1)$ in this definition, but the advantage of our choice is that  $1-\rho_1$ becomes  $\ka$-homogeneous of degree  zero. Clearly, $1-\rho_1$ is nevertheless  supported in a region of the form $D_\ext.$ 

\begin{prop}\label{maxoutroot}
If $p>d=d(\phi),$ and if the neighborhood $\Om$ of $(0,0)$ is chosen sufficiently small,  then the maximal operator $\M^{1-\rho_1}$
is bounded on $L^p$.
\end{prop}
Since $h\ge d,$ this result  implies in particular the $L^p$-boundedness of $\M^{1-\rho_1}$ when $p>h.$ 

\begin{proof}
We can proceed exactly as in Section \ref{adaptedc}, with the measure $\mu$ replaced by $\mu^{1-\rho_1};$  for $\ka$ we still choose the principal weight. The re-scaled measures $\mu_k$ are  then now given by 

$$
\int f\, d\mu_k:=\int f(x,2^k+\phi^k(x))\, \eta(\delta_{2^{-k}}x)\chi_1(x)(1-\rho)\Big(\frac{ x_2-\omega(0)x_1^m}{\varepsilon x_1^{m}}\Big)\,dx,
$$ 
and we have to estimate  oscillatory integrals of the form
$$
J(\xi,\xi_3):= \int e^{i(\xi \cdot x+\xi_{3}\phi^k(x))}\eta(x)(1- \rho)\Big(\frac{ x_2-\om(0)x_1^m}{\ve x_1^{m}}\Big)\,dx,
$$
where $\eta$ is smooth with compact support in the annulus $\A.$ Therefore  the amplitude in this oscillatory integral  $J(\xi,\xi_3)$ is  supported in the intersection of the annulus $\A,$ on which $|x|\sim 1,$ and the  region given by \eqref{Dext}. This is exactly the kind of oscillatory integral  (whose phase has at worst an Airy type $A_2$  singularity)  that we had estimated in our discussion of singularities of type $D$ in  Chapter 3 of \cite{IMmon} (compare pp. 53--54), where we had shown that $J(\xi,\xi_3)=O((1+|\xi_3|)^{-5/6}),$ uniformly  in $k$ for $k$ sufficiently large. This means that we may  choose 
$\ga:=5/6$ in \eqref{je}.

But, since $m\ge 2,$ we have  $d=(2m+1)/(m+1)\ge 5/3>8/5= 1+1/(2\ga),$ and thus we see that $\M$ is $L^p$-bounded whenever $p>d.$
\end{proof}

\subsection{The contribution by the  transition domain}\label{transdom}

 In order to localize to  the transition domain in a smooth fashion, we put 
$$
\tau(x):=\rho\Big(\frac{ x_2-\psi(x_1)}{\ve x_1^{m}}\Big)
(1-\rho)\Big(\frac{ x_2-\psi(x_1)}{Nx_1^{a}}\Big).
$$
Then clearly $\tau$ is supported in a region of the form $E.$ By means of the change to  the adapted  coordinates 
$(y_1,y_2)=(x_1, x_2-\psi(x_1)),$ we then see that the averaging operator $A^\tau_t$ can be written as  
 $$
A^{\tau}_tf(z,z_3)=\int_{\bR^2}
f\big(z_1-ty_1,z_2-t(y_2+\psi(y_1)),
z_3-t(1+\pad(y))\big)\tau^a (y) \,  \eta^a(y)\, dy,
$$
where
$$
\tau^a(y)=\rho\Big(\frac{ y_2}{\ve  y_1^m}\Big)\, 
(1-\rho)\Big(\frac{ y_2}{Ny_1^{a}}\Big),
$$
and $\eta^a$ is a smooth function supported in a sufficiently small neighborhood of the origin.

\begin{prop}\label{maxtrans}
If $p>h=h(\phi),$ and if the neighborhood $\Om$ of $(0,0)$ is chosen sufficiently small,  then the maximal operator $\M^{\tau}$
is bounded on $L^p$.
\end{prop}

 \proof  Following our approach from \cite{IKM-max}, which had been inspired by Phong and Stein's article
   \cite{phong-stein}, we  decompose the domain $E$  (which is a domain of transition between the homogeneities  given be the weight $\ka$  and the weight $\ka^a$)   dyadically in each coordinate separately,
     and then re-scale each of the bi-dyadic pieces obtained in this way.

To this end, consider a dyadic partition of unity $\sum_{k=0}^\infty \chi_k(s)=1$
on the interval $0<s\le 1$ with $\chi \in C_0^\infty(\RR)$ supported
in the interval $ [1/2,4],$ where $\chi_k(s):=\chi(2^ks),$ and put
$$
\chi_{j,k}(x):=\chi_j(x_1)\chi_k(x_2),\quad j,k\in\NN.
$$
We then decompose $A^{\tau}_t$ into the operators
$$
A^{j,k}_tf(z):=\int_{\bR^2}
f\Big(z_1-ty_1,z_2-t(y_2+\psi(y_1)),
z_3-t(1+\pad(y))\Big)\tau^a (y) \, \eta^a(y)\,\chi_{j,k}(y)\, dy,
$$
 with associated maximal operators $\M^{j,k}.$

Notice that by choosing the neighborhood $\Om$ of the origin
sufficiently small, we need only consider sufficiently large $j,k.$
Moreover, because of the localization imposed by $ \tau^a ,$ it
suffices to consider only pairs $(j,k)$ satisfying
\begin{equation}\label{jkcond}
mj+M\le k\le \frac{(n-1)j}2-M,
\end{equation}
where $M$ can still be chosen sufficiently large, because we had the
freedom to choose $\ve$ sufficiently small and $N$ sufficiently
large. In particular, we have $j\sim k,$ and clearly
\begin{equation}\label{MtauMjk}
\|\M^\tau\|_{p\to p}\le \sum_{mj+M\le k\le \frac{(n-1)j}2-M}\|\M^{j,k}\|_{p\to p}.
\end{equation}

By re-scaling in the integral, we have
\begin{eqnarray*}
A_t^{j,k}f(z)
=2^{-j-k} \int_{\bR^2}&& f\Big(z_1-t2^{-j}y_1,z_2-t(2^{-k}y_2+2^{-mj}y_1^m\om(2^{-j}y_1)),\\
&& z_3-t(1+ \pad(2^{-j}y_1,2^{-k}y_2))\Big) \tilde \tau^{j,k}(y) \,
\tilde \eta^{j,k}(y)\, \chi(y_1)\chi(y_2)\, dy,
\end{eqnarray*}
with
$$
\tilde\tau^{j,k}(y):=\rho\Big(\frac{ y_2}{\ve 2^{k-m j}y_1^m}\Big)\,
(1-\rho) \Big(\frac{
y_2}{N2^{k-\frac{(n-1)j}2}y_1^{\frac{n-1}2}}\Big),
  \quad \tilde \eta^{j,k}(y):= \eta^a(2^{-j}y_1,2^{-k}y_2).
$$
Notice that, by \eqref{jkcond}, all derivatives of $\tilde \tau^{j,k}$
are uniformly bounded in $j,k.$

The scaling operators
$$
T^{j,k}f(z):=2^{\frac{(m+2)j+2k}p}f(2^j z_1, 2^{mj} z_2,
2^{j+2k}z_3)
$$ then transform these operators into the averaging operators $\tilde A_t^{j,k}:=T^{-j,-k}A_t^{j,k}T^{j,k},$ i.e., 
\begin{eqnarray*}
\tilde A_t^{j,k}f(z)
=2^{-j-k} &&\int_{\bR^2} f\Big(z_1-ty_1,z_2-t(2^{mj-k}y_2+y_1^m\om(2^{-j}y_1)),\\
&& z_3-t(2^{j+2k}+\tp^{j,k}(y))\Big) \tilde \tau^{j,k}(y) \, \tilde
\eta^{j,k}(y)\, \chi(y_1)\chi(y_2)\, dy,
\end{eqnarray*}
where
$$\tp^{j,k}(y):= 2^{j+2k}\phi^a(2^{-j}y_1,2^{-k}y_2).
$$
Notice that by \eqref{jkcond} we  then have 
$$
\tp^{j,k}(y)=b_1(0,0) y_1y_2^2+O(2^{-j}+2^{-k}+2^{-M}).
$$

In order to simplify notation let us put
\begin{equation}\label{scaledph}
\phi(x,\de):= b(x,\de)x_2^2+\de_3x_1^n
\beta(\de_1x_1),
\end{equation}
where
$$
b(x,\de):=b^a_1(\de_1x_1,\de_2x_2)x_1+\de_4x_2^2b^a_2(\de_2x_2).
$$
Here $\de:=(\de_0,\dots,\de_4)$ is supposed to be very small, i.e., $|\de|\ll 1.$ Then  for the special values 
\begin{equation}\label{dedef1}
\de_0:=2^{mj-k}, \, \de_1:=2^{-j},\,
\de_2:=2^{-k},\, \de_3:=2^{2k-(n-1)j},\, \de_4:=2^{j-2k},
\end{equation}
we find that $\tp^{j,k}(y)=\phi(y,\de).$  Notice also that, due to the condition \eqref{jkcond}, for these  values of $\de$  we have indeed $|\de|\ll 1.$  
\smallskip

From now on we shall then consider the phase as well as the corresponding averaging operators as quantities depending on the non-negative small  perturbation parameters $\de_i$ of which the vector $\de$ is composed, which are otherwise arbitrary. Moreover, we shall denote the variable  $y$ again by $x.$

Let us also introduce an additional parameter $T\ge 0,$ which in our application to the averaging operators $\tilde A_t^{j,k}$ will become
$$T=2^{j+2k}.$$ 
Then, for  these more general sets  of parameters $\de$ and $T,$ we introduce the measure $\nu_{\de,T}$ by putting
$$
\int f d\nu_{\de,T} :=\int f(x_1,\, \de_0x_2+x_1^m\omega(\de_1x_1),\,
T+\phi(x,\de))\eta(x,\de)\chi_1(x_1)\chi_1(x_2)dx,
$$
where $\eta(x,\de):=\eta^a(\de_1x_1,\de_2 x_2),$ and set  $A_{\de,T} f:=f*\nu_{\de,T}.$ By   $\M_{\de,T}$ we denote the  maximal operator corresponding to the averaging operators $(A_{\de,T})_t.$ 

Note that in the considered domain of integration  $|x_1|\sim1,\, |x_2|\sim1,$
and hence also $b(x,\de) \sim1,$ if we assume without loss of generality that $b_1(0,0)=1,$ and  that $\phi(x,0)=x_1x_2^2.$ 

For the particular  choice of $\de$ given by \eqref{dedef1} and $T=2^{j+2k}$, we then have $ \tilde A_t^{j,k}=2^{-j-k}(A_{\de,T})_t,$  so that 
 \begin{equation}\label{Mjk}
\|\M^{j,k}\|_{p\to p}\le 2^{-j-k}\|\M_{\de,T}\|_{p\to p}.
\end{equation}

It will thus suffice to estimate the maximal operator $\M_{\de,T}.$ Note, however,  that
 in view of  \eqref{jkcond}, in our application, where  $\de$ is given by  \eqref{dedef1} and $T=2^{j+2k},$ we have  that   
\begin{equation}\label{Test}
T\ge \de_0^{-2}, \quad \mbox{i.e., }\ \de_0\ge T^{-\frac 12}.
\end{equation}
We shall therefore  assume  that this relation between $\de_0$ and $T$ holds true also in  the subsequent  study of the maximal operator $\M_{\de,T},$  as well as that $|\de|\ll 1,$  for otherwise general $\de$ and $T.$ Notice that that this does not effect our definition of $\nu_{\de,T'}$ for $T'=0.$ 
\medskip

To begin with, notice that the  function $\phi(x,\de)$ is a small 
perturbation of the function $\phi(x,0)=x_1x_2^2,$ so that in the limit as $\de\to 0$ the limiting measure $\nu_0$ is supported in the hypersurface given  by all points  $(x_1, \om(0)x_1^m, T+x_1x_2^2)$  with $|x_1|\sim 1\sim |x_2|.$ Choosing $y_1=x_1$ and $y_2=x_1x_2^2$ as  new coordinates for this hypersurface, we see that it has exactly one non-vanishing principal curvature at every point, so that an application  of Proposition \ref{simple} (with $\ga=1/2$) would only allow to control the associated maximal operator for $p>2.$ We therefore must apply a more refined analysis  and shall invoke ideas as well as notation  from \cite{IMmon} (see, e.g., Section 4.1)  based on additional  dyadic decompositions in every  frequency  variable. This analysis will allow us to  take advantage of the lower bound for $\de_0$ given by \eqref{Test}. 

 To this end, we fix again   suitable smooth cut-off functions $\chi_l\ge 0$ on $\RR$ as in the proof of Proposition \ref{simple}  such that for $l\ge 1,$ $\chi_l(t)=\chi_1(2^{1-l}t)$ is supported where $|t|\sim 2^l$ and 
$$
\sum_{l=0}^\infty\chi_l(t)=1\quad \mbox{ for all } t\in\RR,
$$
and define for every multi-index $l=(l_1,l_2,l_3)\in\NN^3$ the cut-off function
\begin{equation}\nonumber
\chi_l(\xi):=\chi_{l_1}(\xi_1)\chi_{l_2}(\xi_2)\chi_{l_3}(\xi_3).
\end{equation}
In the sequel, we shall usually write $\la=(\la_1,\la_2,\la_3)$ in place of $(2^{l_1-1}, 2^{l_2-1},2^{l_3-1}),$ and define 
accordingly the complex measures $\nu_{\de,T}^\la$ by  
$$
\widehat{\nu_{\de,T}^\la}(\xi):=\chi_l(\xi)\widehat{\nu_{\de,T}}(\xi).
$$
Notice that if  $l_i\ge1$ for $i=1,2,3,$ then 
$$
\widehat{\nu_{\de,T}^\la}(\xi)=\chi_1\left(\frac{\xi_1}{\la_1}\right)\chi_1\left(\frac{\xi_2}{\la_2}\right)\chi_1\left(\frac{\xi_3}{\la_3}\right)
\widehat{\nu_{\de,T}}(\xi),
$$
and 
\begin{equation}\label{5.13+}
|\xi_i| \sim\la_i \quad \mbox{  on  } \supp\widehat{\nu_{\de,T}^\la}.
\end{equation}
 We then find that, in the sense of distributions, 
\begin{equation}\label{5.12+}
\nu_{\de,T}=\sum_{\la} \nu_{\de,T}^\la,
\end{equation}
where the summation $\sum_{\la} $  will always mean summation over the set $\Lambda$ of all {\it dyadic} $\la$ with  $\la_i\ge 2^{-1}$ for $i=1,2,3.$ 
 To simplify the subsequent discussion, we shall concentrate on those measures $\nu_{\de,T}^\la$ for which none of its components $\la_i$ equals $2^{-1},$ i.e.,  $l_i\ge 1,$ since the remaining cases where some  $l_i=0$  can be dealt with in the same way as the corresponding cases where $l_i\ge 1$ is small.
By   $\M_{\de,T}^\la$ we denote  the maximal operator corresponding to the convolution  operators $(A_{\de,T}^{\la})_tf:=f*(\nu_{\de,T}^{\la})_t, \ t>0.$  

The  Fourier transform of $\nu_{\de,T}$ is explicitly given by
$$
\widehat{\nu_{\de,T}} (\xi)=\int e^{-i\Phi(x,\de,T,\xi)}\eta(x,\de)\chi_1(x_1)\chi_1(x_2)dx,
$$
with complete phase 
$$
\Phi(x,\de,T,\xi):=\xi_1x_1+\xi_2(\de_0x_2+x_1^m\omega(\de_1x_1))+\xi_3(T+\phi(x,\de)).
$$

 The following lemma  will be used frequently:
 
 \begin{lemma}\label{maxla}
 \begin{itemize}
\item[(a)]   The maximal operator $\M_{\de,T}^\la$ is of weak-type $(1,1),$ with norm bounded by 
$$
 \|\M_{\de,T}^\la\|_{L^1\to L^{1,\infty}}\le C T \|\nu_{\de,0}^\la\|_\infty.
 $$
 where the constant $C$ is independent of $\de$ and $T.$
\item[(b)]  For $1<p\le2,$ the maximal operator $\M_{\de,T}^\la$ is bounded on $L^p$ with norm controlled by 
$$
 \|\M_{\de,T}^\la\|_{p\to p}\le C_p T^{\frac 2p-1} (\la_3T+\la_1+\la_2)^{1-\frac 1p} \|\nu_{\de,0}^\la\|_\infty^{\frac 2p -1}  \|\widehat{\nu_{\de,0}^\la}\|_\infty^{2-\frac 2p}.
 $$
 \end{itemize}
\end{lemma}
\begin{proof}  The proof follows the pattern of the proof of Proposition \ref{simple}.
Indeed, arguing in the same way as in that proof and observing that still  the measures $\nu_{\de,T}^\la$ are  essentially supported in a  cuboid of dimensions comparable to 
$1\times 1\times T,$ we see that $\M_{\de,T}^\la f$ is dominated  by $CT \|\nu_{\de,0}^\la\|_\infty M_{HL}(|f|),$ which implies (a). 

Moreover, using again Littlewood-Paley theory, we also easily see that 
$$\|\M_{\de,T}^\la\|_{2\to 2}\le C_p(\la_3T+\la_1+\la_2)^{\frac 12}   \|\widehat{\nu_{\de,0}^\la}\|_\infty.
$$
(b) follows then again by an application of  Marcinkiewicz's interpolation theorem.
\end{proof} 
 
In order to estimate $ \|\nu_{\de,0}^\la\|_\infty,$ we write 
\begin{eqnarray}\nonumber
\nu^{\la}_{\de,0}(x)=\la_1\la_2\la_3&\int \check\chi_1\big({\la_1}(x_1-y_1)\big) \, \check\chi_1\big({\la_2}(x_2-\de_0y_2-y_1^m\omega(\de_1y_1))\big) \label{5.14+}\\
&\check\chi_1\big({\la_3}(x_3-\phi(y,\de))\big) \,\eta(y,\de)\,
\chi(y_1)\chi(y_2)\, dy_1dy_2,
\end{eqnarray}
where $\check f$ denotes the inverse Fourier transform of $f.$  Observe that   $|\pa_{y_2}\phi(y,\de)|\sim 1,$
so that $ (y_1,\phi(y_1,y_2,\de))$ can be used as coordinates in place of $(y_1,y_2).$  We may  therefore change  coordinates from $(y_1,y_2)$ to $(u_1,u_2)$ in this integral, where  $y_1=u_1/\la_1$ and  $\phi(y,\de)= u_2/\la_3,$  which easily leads to the uniform estimate $|\nu^{\la}_\de(x)|\le C \la_2,$ with $C$ independent of $x,\de$ and $\la.$ 
 Similarly, the change of coordinates $y_1=u_1/\la_1$ and  $y_2= u_2/(\de_0\la_2)$ leads to $|\nu^{\la}_\de(x)|\le C \la_3/\de_0.$ Altogether, we arrive at  the  uniform estimate
 \begin{equation}\label{nulainfty1}
 \|\nu_{\de,0}^\la\|_\infty\lesssim  \min \{\la_2, \la_3 \de_0^{-1}\}.
\end{equation}
  
 Recall also that $\phi(x,\de)=x_1x_2^2+O(\de),$  and that we are interested in exponents $2>p>h\ge 3/2.$ 
 
  We shall distinguish six cases depending on the relative sizes of $\la_1,\la_2,\la_3$ and $\de_0,$ and shall accordingly decompose the set $\Lambda$ of our dyadic $\la$'s   into subsets  $I_i,$  where $I_i$ will correspond to Case $i.$  It will therefore be convenient to use the following notation: for any given subset  $I\subset \Lambda,$ we let   $A_{\de,T}^{I}:= \sum_{\la\in I} A_{\de,T}^\la$
  denote the contribution to $A_{\de,T}$ by the operators $A_{\de,T}^\la$ with $\la\in I,$ with associated maximal operator $\M_{\de,T}^{I}.$ Then clearly 
  $$
 \| \M_{\de,T}^{I}\|_{p\to p}\le \sum_{\la \in I} \| \M_{\de,T}^{\la}\|_{p\to p},
  $$
  and moreover we shall have 
  $$
   \| \M_{\de,T}\|_{p\to p}\le \sum_{i =1}^6\| \M_{\de,T}^{I_i}\|_{p\to p}.
$$
For each of the maximal operators  $\M_{\de,T}^{I_i}$ we shall prove the following estimate:
\begin{equation}\label{MIi}
\| \M_{\de,T}^{I_i}\|_{p\to p}\lesssim T^{\frac 1p}.
\end{equation}
This will then  imply that $\| \M_{\de,T}\|_{p\to p}\lesssim T^{\frac 1p}=2^{\frac{j+2k}p},$  and  combining this with \eqref{Mjk} we see that 
$$
\| \M^{j,k}\|_{p\to p}\lesssim 2^{-j(1-\frac 1p)} 2^{k(\frac 2p -1)}.
$$
By \eqref{MtauMjk} we then obtain the estimate
$$
\|\M^\tau\|_{p\to p}\lesssim \sum_{mj+M\le k\le \frac{(n-1)j}2-M} 2^{-j(1-\frac 1p)} 2^{k(\frac 2p -1)}=\sum_{j\ge 0} 2^{-j(\frac{n+1}2-\frac np)}<\infty,
$$
  since  $p>h=2n/(n+1),$ which will conclude the proof of Proposition \ref{maxtrans}.
   \medskip
   
{\bf  Case 1: $\la_3\gtrsim \max\{\la_1,\la_2\}.$} 
Iterated integrations by parts  in $x_1$ then lead to the estimate
 \begin{equation}\nonumber 
 \|\widehat{\nu_{\de,0}^\la}\|_\infty \lesssim \la_3^{-N} \qquad \mbox{for every} \ N\in \NN.
\end{equation}
Moreover, by \eqref{nulainfty1} we have  $\|\nu_{\de,0}^\la\|_\infty\lesssim  \la_2.$  Observe also that $\la_3T$ is the dominant term in $\la_3T+\la_1+\la_2,$ and thus  Lemma \ref{maxla} implies 
$$
 \|\M_{\de,T}^\la\|_{p\to p}\lesssim T^{\frac 1p}\la_2^{\frac 2p-1} \la_3^{-N}
$$
for every $N\in \NN.$ This easily implies that 
\begin{equation}\label{Msum1}
\|\M_{\de,T}^{I_1}\|_{p\to p}\le \sum_{\la_3\gtrsim \max\{\la_1,\la_2\}}\|\M_{\de,T}^\la\|_{p\to p}\lesssim T^{\frac 1p}.
\end{equation}
 The constants in these estimate do not depend on $\de.$ 
 \smallskip

\medskip
{\bf  Case 2: $\la_1\gg \max\{\la_2,\la_3\}.$} 
Iterated integrations by parts  in $x_1$  here  lead to the estimate
 \begin{equation}\nonumber 
 \|\widehat{\nu_{\de,0}^\la}\|_\infty \lesssim \la_1^{-N} \qquad \mbox{for every} \ N\in \NN.
\end{equation}
Moreover, by \eqref{nulainfty1} we have  $\|\nu_{\de,0}^\la\|_\infty\lesssim  \la_2.$  Observe also that $\la_3T+\la_1+\la_2\lesssim  \la_1T.$   From here on we can proceed as in the previous case, with the roles of $\la_3$ and $\la_1$ interchanged, and in analogy with \eqref{Msum1} arrive at 
$$
\|\M_{\de,T}^{I_2}\|_{p\to p}\lesssim T^{\frac 1p}.
$$

   \medskip
   
{\bf  Case 3: $\la_2\gg  \max\{\la_1,\la_3\}.$} 
Then again iterated integrations by parts  in $x_1$  lead to the estimate
 \begin{equation}\nonumber 
 \|\widehat{\nu_{\de,0}^\la}\|_\infty \lesssim \la_2^{-N} \qquad \mbox{for every} \ N\in \NN.
\end{equation}
Also, by  \eqref{nulainfty1} we have  $\|\nu_{\de,0}^\la\|_\infty\lesssim  \la_2,$   and moreover clearly $\la_3T+\la_1+\la_2\lesssim \la_2T.$ Thus  Lemma \ref{maxla} implies 
$$
 \|\M_{\de,T}^\la\|_{p\to p}\lesssim T^{\frac 1p}\la_2^{-N}
$$
for every $N\in \NN,$ so that  
$$
\|\M_{\de,T}^{I_3}\|_{p\to p}\lesssim T^{\frac 1p}.
$$

 \medskip
 There remain those cases where $\la_1\sim \la_2\gg \la_3.$

\medskip
   
{\bf  Case 4: $ \la_1\sim \la_2$ and $ \la_2\de_0\ll \la_3\ll \la_2.$} 
Here, we may first integrate by parts  in $x_2$ $N$-times, and then either apply the method of stationary phase in $x_1$ or again integrate by parts in $x_1$ (in case that there is no critical point) in order to see that 
 \begin{equation}\nonumber 
 \|\widehat{\nu_{\de,0}^\la}\|_\infty \lesssim \la_3^{-N} \la_2^{-\frac 12}\qquad \mbox{for every} \ N\in \NN.
\end{equation}
Moreover, by \eqref{nulainfty1} we have  $\|\nu_{\de,0}^\la\|_\infty\lesssim  \la_2.$  Observe also that $\la_3T$ is the dominant term in $\la_3T+\la_1+\la_2,$ since, by \eqref{Test},
$$
\la_3T\gg \la_2T\de_0>\la_2\de_0^{-1}\gg \la_2. 
$$
Therefore   Lemma \ref{maxla} implies 
$$
 \|\M_{\de,T}^\la\|_{p\to p}\lesssim T^{\frac 1p}\la_2^{-2+\frac 3p} \la_3^{-N}
$$
for every $N\in \NN.$  Since $p>3/2,$ this shows that we can sum these estimates over the $\la\in I_4$ and obtain
$$
\|\M_{\de,T}^{I_4}\|_{p\to p}\lesssim T^{\frac 1p}.
$$

\medskip
{\bf  Case 5: $ \la_1\sim \la_2$ and $\la_3\ll \la_2\de_0.$} 
Arguing in the same way as in the previous case we here find that 
 \begin{equation}\nonumber 
 \|\widehat{\nu_{\de,0}^\la}\|_\infty \lesssim (\la_2\de_0)^{-N} \la_2^{-\frac 12}\qquad \mbox{for every} \ N\in \NN.
\end{equation}
Moreover, by \eqref{nulainfty1} we have  $\|\nu_{\de,0}^\la\|_\infty\lesssim  \la_3\de_0^{-1}.$  Therefore   Lemma \ref{maxla} implies that 
$$
 \|\M_{\de,T}^\la\|_{p\to p}\lesssim T^{\frac 2p-1}(\la_3T+\la_2)^{1-\frac 1p} (\la_3\de_0^{-1})^{\frac 2p-1} ((\la_2\de_0)^{-N} \la_2^{-\frac 12})^{2-\frac 2p}
$$
for every $N\in \NN.$ Since $\la_3T\ll \la_2T\de_0$ and, as before, $\la_2T\de_0\gg \la_2,$ we may control 
 $\la_3T+\la_2$ by $\la_2T\de_0,$ so that 
$$
 \|\M_{\de,T}^\la\|_{p\to p}\lesssim T^{\frac 1p} \de_0^{2-\frac 3p}  \la_3^{\frac 2p-1}(\la_2\de_0)^{-N}
$$
for every $N\in \NN.$ Summing first over all $\la_1\sim \la_2,$ then all $\la_2$ such that $\la_2\de_0\gg \la_3$ and finally over all $\la_3,$ we then find that 
$$
\|\M_{\de,T}^{I_5}\|_{p\to p}\lesssim T^{\frac 1p}\de_0^{2-\frac 3p}\le T^{\frac 1p},
$$
the last inequality being true since $p>3/2.$

\medskip
{\bf  Case 6: $ \la_1\sim \la_2$ and $\la_3\sim\la_2\de_0.$} In this case, there is possibly a (non-degenerate) critical point $x_2^c=x_2^c(x_1,\de,\xi)$ of the phase with respect to $x_2$ in the region where $|x_2|\sim 1.$ In that case, we apply the  method of stationary phase to the integration in $x_2,$ which leads to an oscillatory integral in $x_1$ whose phase is of the form 
$$
\xi_1x_1+\xi_2(\de_0x_2^c+x_1^m\omega(\de_1x_1))+\xi_3(x_1(x_2^c)^2+O(\de)).
$$
But, since $|\xi_3|\ll \max\{|\xi_1|,|\xi_2|\}$ under our assumptions,  the last term can be viewed as a small error term, and thus we may apply van der Corput's estimate of order $2$ to the remaining integral in $x_1$ and altogether arrive at the estimate
 \begin{equation}\nonumber 
 \|\widehat{\nu_{\de,0}^\la}\|_\infty \lesssim \la_3^{-\frac 12} \la_2^{-\frac 12}.
\end{equation}
If there is no critical point with respect to $x_2,$ integrations by parts in $x_2$ in place of an application of the method of stationary  phase lead to even better estimates. Moreover,  by \eqref{nulainfty1} we have  $\|\nu_{\de,0}^\la\|_\infty\lesssim  \la_2,$  and as in Case 4 we have $\la_3T\gg \la_2.$  Therefore   Lemma \ref{maxla} implies that 
$$
 \|\M_{\de,T}^\la\|_{p\to p}\lesssim T^{\frac 1p}\la_3^{1-\frac 1p}\la_2^{\frac 2p-1}( \la_3^{-\frac 12} \la_2^{-\frac 12})^{2-\frac 2p}=T^{\frac 1p}\la_2^{\frac 3p-2}.
$$
Since $p>3/2,$ we see that the sum over all indices $\la\in I_6$ considered in this case is finite, and we obtain 
$$
\|\M_{\de,T}^{I_6}\|_{p\to p}\lesssim T^{\frac 1p}.
$$

This finishes  the proof of Proposition \ref{maxtrans}.
\qed
\medskip

\subsection{The contribution by the region  near the principal root jet}\label{nearjet}

Finally, we consider the contribution to  the maximal operator $\M$ by the region 
$$ 
D_\pr:=\{x\in \Om:|x_2-\psi(x_1)|\le N|x_1|^a\}
$$
close to the principal root jet ($N$ is a fixed positive number). We localize to a region of this type by means of  the cut-off function
$$
\rho_2(x):=\rho\Big(\frac{ x_2-\psi(x_1)}{Nx_1^{a}}\Big).
$$
\begin{prop}\label{maxpr}
If $p>h=h(\phi),$ and if the neighborhood $\Om$ of $(0,0)$ is chosen sufficiently small,   the maximal operator $\M^{\rho_2}$
is bounded on $L^p$.
\end{prop}
In combination with Proposition \ref{maxoutroot} and \ref{maxtrans} this will complete the proof of Theorem \ref{s1.2}, with the exception of the case of  $D_4^+$ - type singularities. 

\begin{proof}

In the adapted coordinates $(y_1,y_2)=(x_1, x_2-\psi(x_1))$ the measure $\mu^{\rho_2}$ can  be expressed as 
 $$
 \int f\, d\mu=\int f(y_1,y_2+\psi(y_1), 1+\pad(y))\rho^a(y)\eta^a(y)\, dy,
 $$
 with $\eta^a$  a smooth function with  support  in a sufficiently  small neighborhood $\Om^a$ of the origin as before, and 
 $$
 \rho^a(y):=\rho\Big(\frac{ y_2}{Ny_1^{a}}\Big).
$$
Notice that $\rho^a$ is $\ka^a=(1/n, (n-1)/(2n))$-homogeneous of degree $0.$  Working now in these adapted coordinates $\y,$ we next proceed as in  Section \ref{adaptedc}, only with the weight  $\ka$ replaced by $\ka^a,$ and dilations $\de_r$ replaced by the $\ka^a$-dilations  $\de^a_r$ from Section \ref{nonadac}. Recall to this end from \eqref{padpr} that the principal part of $\pad,$ which is $\ka^a$-homogeneous of degree $1,$ is given by
$$
\pad_\pr\y=y_1y_2^2+y_1^n\be(0),
$$
if we assume again without loss of generality that $b_1(0,0)=1.$ 
 We choose a smooth bump-function $\chi_1$ supported in the annulus $\A$ such that 
$$
\sum_{k=k_0}^\infty \chi_1(\delta^a_{2^k}y)=1 \quad \mbox{for }\quad 0\ne y\in\Om^a,
$$
and decompose 
$$
\mu=\sum_{k=k_0}^\infty\tilde\mu_k,
$$
with 
$$
\int f\, d\tilde\mu_k:=\int f(y_1,y_2+y_1^m\om(y_1), 1+\pad(y))\rho^a(y)\eta^a(y)\chi_1(\delta^a_{2^k}y)\, dy.
$$
It will then suffice to derive suitable $L^p$-estimates for the maximal operators $\sup_{t>0}|f*(\tilde\mu_k)_t|.$ 
Applying a straight-forward $L^p$-isometric re-scaling to them by means of the dilations $\de^a_{2^{-k}},$ we may assume that these are of the  form $2^{-|\ka^a| k}\M_k f,$ where 
$$
\M_kf(y,y_3):=\sup_{t>0}|f*(\mu_k)_t(y,y_3)|
$$
and 
$$
\int f\, d\mu_k:=\int f\big(y_1,2^{-(\ka^a_2-m\ka^a_1)k}y_2+y_1^m\omega(2^{-\ka_1^ak}y_1), 2^k+\phi^k(y)\big)\rho\Big(\frac{ y_2}{Ny_1^{a}}\Big)\eta^a(\delta^a_{2^{-k}}y)\chi_1(y)\, dy.
$$
Here we have set $\phi^k(y):=2^k\pad(\de^a_{2^{-k}}(y)).$ Then  clearly 
$$\phi^k(y)=\pad_\pr(y)+O(2^{-\ve k})=y_1y_2^2+y_1^n\be(0)+O(2^{-\ve k})
$$ 
for some $\ve>0.$ Notice also that $|y_1|\sim 1$ and $|y_2|\lesssim 1$  on the support of $\mu_k.$

Similarly as in Section \ref{adaptedc},  the  following analogue of estimate \eqref{mtomk} holds true: 
\begin{equation}\label{mtomk2}
\|\M^{\rho_2}\|_{L^p\to L^p}\le \sum_{k= k_0}^\infty  2^{-|\ka^a| k}\|\M_k\|_{L^p\to L^p}.
\end{equation}

In order to simplify notation we shall here  write  
\begin{equation}\label{dedef2}
\de_0:=2^{-(\ka^a_2-m\ka^a_1)k},\, \de_1:=2^{-\ka_1^a k},\, \de_2:=2^{-\ka_2^a k} \\,\de_3:=2^{-\frac k{2n}} \quad \mbox{and }\quad T:=2^k,
\end{equation}
and put $\de:=(\de_0,\de_1,\de_2,\de_3).$ Recall that $a=\ka^a_2/\ka^a_1>m,$ so that $|\de|\ll 1.$ 

Observe that as in the previous subsection the relation \eqref{Test} is valid, i.e., 
$$
T\ge \de_0^{-2},
$$
since $2\ka^a_2=(n-1)/n<1,$ so that  $\de_0^{-2}\le 2^{2\ka^a_2k}\le 2^k=T.$ 
\smallskip

We  shall from now on consider the phase as well as the corresponding averaging operators as quantities depending on the non-negative perturbation parameters $\de_i$ of which the vector $\de$ is composed (the phase $\phi^k(y)$ will indeed be viewed as a function $\phi(y,\de)$ depending only on $y$ and the  ``dummy'' parameter $\de_3.$)  These  are assumed to be sufficiently small.    Accordingly, we shall re-write the measure $\mu_k$ as $\nu_{\de,T},$ where $\nu_{\de,T}$ is of the form
$$
\int f\, d\nu_{\de,T}:=\int f\big(y_1,\de_0y_2+y_1^m\omega(\de_1y_1), T+\phi(y,\de)\big)\eta(y,\de)\chi_0(y_2)\chi_1(y_1)\, dy,
$$
where $\phi(y,\de)$ and  $\eta(y,\de)$ are smooth functions in $y$ and $\de$, and where $\chi_0$ is smooth and  supported in a compact neighborhood of $0,$ whereas as before  $\chi_1(y_1)$ is supported where $|y_1|\sim 1.$  Moreover, 
$$
\phi(y,0) =y_1y_2^2+y_1^n\be(0).
$$
As in the previous Subsection \ref{transdom}, the corresponding averaging operator will be denoted by $A_{\de,T},$ with corresponding maximal operator $\M_{\de,T}.$ Then, in analogy with \eqref{Mjk}, we have 
$$
\|\M^{k}\|_{p\to p}\le \|\M_{\de,T}\|_{p\to p},
$$
if $\de$  and $T$ are   given by \eqref{dedef2}. We shall prove the following uniform estimate 
\begin{equation}\label{Mest4}
 \|\M_{\de,T}\|_{p\to p}\lesssim T^{\frac 1p},
\end{equation}
provided that $p>12/7.$  In combination with \eqref{mtomk2} this will imply that 
\begin{equation}\label{mtomk3}
\|\M^{\rho_2}\|_{L^p\to L^p}\lesssim \sum_{k= k_0}^\infty  2^{-|\ka^a| k}2^{\frac kp}<\infty,
\end{equation}
if $p>h=1/|\ka^a|$  and hence conclude the proof of Proposition \ref{maxpr}.

For the proof of  \eqref{Mest4} we follow our approach  from  the previous subsection. Using also  the same notation that we had introduced therein, we perform an additional dyadic frequency decomposition by putting 
$$
\widehat{\nu_{\de,T}^\la}(\xi)=\chi_1\left(\frac{\xi_1}{\la_1}\right)\chi_1\left(\frac{\xi_2}{\la_2}\right)\chi_1\left(\frac{\xi_3}{\la_3}\right)
\widehat{\nu_{\de,T}}(\xi),
$$
where 
$$
\widehat{\nu_{\de,T}} (\xi)=\int e^{-i\Phi(y,\de,T,\xi)}\eta(y,\de)\chi_0(y_2)\chi_1(y_1)\, dy,
$$
with complete phase 
\begin{eqnarray}\nonumber
\Phi(y,\de,T,\xi)&:=&\xi_1y_1+\xi_2(\de_0y_2+y_1^m\omega(\de_1y_1))+\xi_3(T+\phi(y,\de)) \\\label{phicompl2}
&=&\xi_1y_1+\xi_2(\de_0y_2+y_1^m\omega(\de_1y_1))+\xi_3(T+y_1y_2^2+y_1^n\be(0)+O(\de)).
\end{eqnarray}
Notice, however, that in contrast to the previous subsection, in this integral we have 
$$|y_1|\sim 1\quad \mbox {and} \quad  |y_2|\lesssim 1.
$$ 
By $\M_{\de,T}^\la,$  we denote the maximal operator defined by the dilates of $\nu_{\de,T}^\la.$

In analogy with  \eqref{5.14+}, we have  
\begin{eqnarray*}
|\nu^{\la}_{\de,0}(x)|\le \la_1\la_2\la_3&\int \big|\check\chi_1\big({\la_1}(x_1-y_1)\big) \, \check\chi_1\big({\la_2}(x_2-\de_0y_2-y_1^m\omega(\de_1y_1))\big) \\
&\check\chi_1\big({\la_3}(x_3-\phi(y,\de))\big) \,\eta(y,\de)\,
\chi_0(y_2)\chi_1(y_1)\Big|\, dy_1dy_2.
\end{eqnarray*}
We can estimate this by means of the same type of arguments that we used in \cite{IMmon}. Indeed, for $y_1$ fixed, we can first make use of the localization given by the third factor in this integral and apply the van der Corput-type Lemma 2.1 (b) of order $N=2$ in \cite{IMmon} to see that 
$$\int  |\check\chi_1\big({\la_3}(x_3-\phi(y,\de))\big)| \,dy_2\lesssim \la_3^{-\frac 12},
$$
uniformly in $y_1,x$ and $\de.$  Subsequently we may estimate the remaining integral in $y_1$ by performing the change of variables  $y_1\mapsto y_1/\la_1,$  which gains another factor $\la_1^{-1}.$ Altogether, this leads to  the  uniform estimate
 \begin{equation}\label{nulainfty2}
 \|\nu_{\de,0}^\la\|_\infty\lesssim  \la_2\la_3^{\frac 12}.
\end{equation}

 Again we may and shall apply Lemma \ref{maxla} and proceed by distinguishing  here five   cases, assuming always that $p>12/7.$ 

   \medskip
   
{\bf  Case 1: $\la_3\gtrsim \max\{\la_1,\la_2\}.$} 
For $\de=0$ and $T=0$ the complete phase is given by 
$$
\Phi(y,0,0,\xi):=\xi_1y_1+\xi_2y_1^m\omega(0)+\xi_3(y_1y_2^2+y_1^n\be(0)).
$$
We may here argue in a similar way as in Case 6 of the previous subsection. The  complete phase has a  non-degenerate critical point $y_2^c=0$ at the origin, so that we can apply the  method of stationary phase to the integration in $y_2.$ This  leads to an oscillatory integral in $y_1$ whose phase is given by 
$$
\xi_1y_1+\xi_2y_1^m\omega(0)+\xi_3y_1^n\be(0).
$$
To the remaining oscillatory integral in $y_1$  we may thus apply the version of  Corput's estimate from Lemma 2.1 (a) in \cite{IMmon}, of order $N=3,$ and altogether arrive at the estimate
 \begin{equation}\nonumber 
 \|\widehat{\nu_{\de,0}^\la}\|_\infty \lesssim \la_3^{-\frac 12} \la_3^{-\frac 13}=\la_3^{-\frac 56}
\end{equation}
for $\de=0.$ The argument is stable under small perturbations,  and thus this estimate remains valid for sufficiently small $\de.$  In combination with \eqref{nulainfty2} we may then conclude by means of    Lemma \ref{maxla}  that 
$$
 \|\M_{\de,T}^\la\|_{p\to p}\lesssim T^\frac1{p} \la_3^{1-\frac1{p}}( \la_2\la_3^{\frac 12})^{\frac2{p}-1}\la_3^{-\frac56(2-\frac2{p})}.
$$
Thus 
$$
\sum_{\la_1,\la_2\lesssim \la_3}\|\M_{\de,T}^\la\|_{p\to p}\lesssim T^\frac1{p}(\log \la_3) \la_3^{\frac{11}{3p}-\frac{13}6}.
$$
We can sum the last inequality in $\la_3$ provided  $p>{22}/{13}.$ In particular, since  $p>{12}/7>{22}/{13},$ we find that  the sum over all indices $\la\in I_1$ considered in this case is finite, and we obtain 
$$
\|\M_{\de,T}^{I_1}\|_{p\to p}\lesssim T^{\frac 1p}.
$$

\medskip
{\bf  Case 2: $\la_1\gg \max\{\la_2,\la_3\}.$} This case can be handled exactly as the corresponding case in the previous subsection  by means of iterated integrations by parts  in $y_1,$  and we easily get  for $p>12/7$
$$
\|\M_{\de,T}^{I_2}\|_{p\to p}\lesssim T^{\frac 1p}.
$$

   \medskip
   
{\bf  Case 3: $\la_2\gg  \max\{\la_1,\la_3\}.$} 
Also this case can be handled exactly as the corresponding case in the previous subsection, and we get for $p>12/7$ 
$$
\|\M_{\de,T}^{I_3}\|_{p\to p}\lesssim T^{\frac 1p}.
$$
 \medskip
   
{\bf  Case 4: $ \la_1\sim \la_2$ and $ \la_2\de_0\lesssim  \la_3\ll \la_2.$} 
In this case we have non-degenerate critical points in $y_2$ and $y_1$ as well. More precisely, applying first the method of stationary phase   to the integration in $y_2,$ and subsequently to the $y_1$-integration, we find that \begin{equation}\nonumber 
 \|\widehat{\nu_{\de,0}^\la}\|_\infty \lesssim \la_1^{-\frac12}\la_3^{-\frac12}.
 \end{equation}
  As in Case 4 of the previous subsection, $\la_3T$ is the dominant term in $\la_3T+\la_1+\la_2,$ and therefore   Lemma \ref{maxla}  implies  that 
$$
 \|\M_{\de,T}^\la\|_{p\to p}\lesssim T^\frac1{p} \la_3^{1-\frac1{p}} (\la_2\la_3^\frac12)^{(\frac2{p}-1)}(\la_2^{-\frac12}\la_3^{-\frac12})^{(2-\frac2{p})}= T^\frac1{p}\la_3^{\frac1{p}-\frac12}\la_2^{\frac3{p}-2}.
$$
Since  the exponent of $\la_3$  in this estimate is a positive real number, we can sum over all $\la_3\ll \la_2$ and obtain
$$
\sum_{\{\la_1,\la_3: \la_1\sim \la_2, \la_3\ll \la_2\}} \|\M_{\de,T}^\la\|_{p\to p}\lesssim 
T^\frac1{p}\la_2^{\frac4{p}-\frac52}.
$$
The expression on the right-hand side can be summed over  all $\la_2$ provided  $p>8/5.$ Since  ${12}/7>8/5,$ we conclude that for $p>12/7$ we have 
$$
\|\M_{\de,T}^{I_4}\|_{p\to p}\lesssim T^{\frac 1p}.
$$

\medskip
{\bf  Case 5: $ \la_1\sim \la_2$ and $\la_3\ll \la_2\de_0.$} 
Arguing in the same way as in the corresponding Case 5 of the previous subsection, we  find that 
 \begin{equation}\nonumber 
 \|\widehat{\nu_{\de,0}^\la}\|_\infty \lesssim (\la_2\de_0)^{-N} \la_2^{-\frac 12}\qquad \mbox{for every} \ N\in \NN.
\end{equation}
  Therefore   Lemma \ref{maxla} implies that 
$$
 \|\M_{\de,T}^\la\|_{p\to p}\lesssim T^{\frac 2p-1}(\la_3T+\la_2)^{1-\frac 1p} (\la_2\la_3^{\frac 12})^{\frac 2p-1} ((\la_2\de_0)^{-N} \la_2^{-\frac 12})^{2-\frac 2p}
$$
for every $N\in \NN.$ As before,  we may control 
 $\la_3T+\la_2$ by $\la_2T\de_0,$ so that 
$$
 \|\M_{\de,T}^\la\|_{p\to p}\lesssim T^{\frac 1p} \de_0^{2-\frac 3p}  \la_3^{\frac 1p-\frac 12}(\la_2\de_0)^{-N}
$$
for every $N\in \NN.$ Summing first over all $\la_1\sim \la_2,$ then all $\la_2$ such that $\la_2\de_0\gg \la_3$ and finally over all $\la_3,$ we then find that 
$$
\|\M_{\de,T}^{I_5}\|_{p\to p}\lesssim T^{\frac 1p}\de_0^{2-\frac 3p}\le T^{\frac 1p},
$$
the last inequality being true if $p>3/2,$ hence in particular for $p>12/7.$ 
\end{proof} 

\medskip

What remains is the study if $D_4^+$ -  type singularities. As it turns out,  this will indeed require an understanding also of  maximal functions associated to surfaces with $A_2$ - type singularities, where exactly one of the principal curvatures of $\phi$    does not vanish at the origin. This (deeper) study will be carried out in the next section.

\setcounter{equation}{0}
\section{Maximal operators associated to families of surfaces with $A_2$ - type singularities depending on small parameters}\label{ndairy}

Consider a smooth family of real-valued functions $\x\mapsto \phi(x_1, x_2, \sigma)$ defined on a given open neighborhood $U$ of the origin in $\RR^2$ and depending smoothly  on parameters $\sigma$ from a given open neighborhood  $V$ of the origin in $\RR^l, $   such that 
\begin{equation}\label{genricform}
\phi(x_1, x_2, 0)=a_2x_2^2+a_3x_1^3+\phi_r(x_1, x_2),
\end{equation}
where $a_2, a_3$ are non-zero real numbers and $\phi_r(x_1, x_2)$ is a smooth function whose Newton polyhedron satisfies  $\N(\phi_r)\subset \{t_1/3+t_2/2>1\}$ (so that $a_2x_2^2+a_3x_1^3$ is the principal part of $\phi$ when $\si=0$). 

The  goal for  this section will be to prove the following 

\begin{thm}\label{nondegairypar}
Denote by $\M^\si_T$ the maximal operator
\begin{equation}\label{masi}
\M^\si_T f(y,y_3):=\sup_{t>0}\Big|\int_{\bR^2} f(y-tx, y_3-t(T+\phi(x,\si)))\eta(x,\si) \, dx\Big |,
\end{equation}
where $T\ge 1,$ and 
where $\eta$ is a smooth, non-negative  function supported in $U\times V.$ Then, if we  assume that the support of $\eta$ is contained in a sufficiently small neighborhood of the origin,    the 
maximal operators $\M^\si_T$ are  uniformly bounded on $L^p(\RR^3)$ in $\si,$ for any given $p>3/2,$ with norm 
$$
\|\M^\si_T\|_{p\to p}\le C_{\de,p} T^{\frac 1p+\de},
$$
for every given $\de>0.$ 
\end{thm}
\begin{remark}\label{notorigin}
{\rm As our proof will show, the same result holds true even if $U$ is a small neighborhood of any  point of distance $\lesssim 1$ to the origin and $T\gg 1$ sufficiently  large.  This will become important to our application of the theorem in the last Section \ref{D4est}.}
\end{remark}

Our proof will rely on the following result on normal forms,   which is a parameter dependent version  of the analogous result  for singularities of type $A_2$ in  Proposition 2.11 of \cite{IMmon}:


\begin{lemma}\label{simpar}
Assume that $\phi_r\equiv 0$ in \eqref{genricform}. 
By  restricting ourselves to sufficiently small open neighborhoods $U$ of $(0,0)$ in $\RR^2$ and $V$ of the origin in the parameter space $\RR^l,$ we can find  affine-linear  coordinates depending  smoothly on $\si$ so that,  in these new coordinates, the function 
$\phi(x_1, x_2, \sigma)$ can be written in the form
\begin{equation}\label{simparform}
\phi(x_1, x_2, \sigma)=b(x_1, x_2, \sigma)(x_2-x_1^m\omega(x_1,\sigma))^2+x_1^3\beta(x_1, \sigma)+\beta_1(\sigma)x_1+\beta_0(\sigma),
\end{equation}
where  $b, \beta, \beta_0, \beta_1$ and $\om$  are smooth functions and $m\ge2$ is a positive integer,
 such that the following hold true:
 \begin{equation}\label{eqnew1}
b(x_1,x_2, 0)=a_2\neq 0, \;\beta(x_1,0)=a_3\neq 0, \beta_0(0)=\beta_1(0)=0, \ \mbox{and  }�\  \om(x_1,0)=0.
\end{equation}
 \end{lemma}

\begin{proof} Following the proof of Proposition 2.11 in \cite{IMmon}, we consider the equation
\begin{equation}\label{shiteq}
\pa_2\phi(x_1, x_2, \sigma)=0.
\end{equation}

Since $\pa_2^2\phi(0, 0, 0)\neq0,$ the implicit function theorem shows that locally near $(0,0,0),$  this equation  has a unique, smooth solution  $x_2=\psi(x_1, \sigma)$ with $\psi(0, 0)=0$. In fact, since $\phi_r\equiv 0,$ we even have  $\psi(x_1, 0)=0.$  A Taylor series expansion of  the function $\phi(x_1, x_2,\sigma)$ with respect to the variable $x_2$  around $\psi(x_1,\si)$ then reveals  that 
\begin{equation}\label{divide}
\phi(x_1, x_2, \sigma)=b(x_1, x_2, \sigma)(x_2-\psi(x_1, \sigma))^2+b_0(x_1, \sigma),
\end{equation}
where $b, b_0$ are smooth functions satisfying the conditions $b(x_1, x_2, 0)=a_2\neq0$ and $b_0(0, 0)=\pa_1 b_0(0, 0)=\pa_1^2b_0(0, 0)=0$ and $\pa_1^3b_0(x_1, 0)=6 a_3\neq0$.

Again by the implicit function theorem, we then see that the equation  $\pa_1^2b_0(x_1, \sigma)=0$ locally near $\si=0$  has a smooth solution $x_1=x_1(\sigma)$ with $x_1(0)=0$. Applying next the change of coordinates  $x_1\mapsto x_1+x_1(\sigma),$ we thus see that we may  assume that $\phi$ is of the form
$$
\phi(x_1, x_2, \sigma)=b(x_1+x_1(\sigma), x_2, \sigma)(x_2-\psi(x_1+x_1(\sigma), \sigma))^2+x_1^3\beta(x_1, \sigma)+\beta_1(\sigma)x_1+\beta_0(\sigma),
$$
with smooth functions $\beta, \beta_1, \beta_0$ satisfying $\beta(x_1, 0)=a_3\neq0, \beta_1(0)=0, \beta_0(0)=0.$

By means of a   Taylor series expansion, we also  see that we can write 
$$
\psi(x_1+x_1(\sigma), \sigma)=\psi(x_1(\sigma), \sigma)+\pa_1\psi(x_1(\sigma), \sigma)x_1+x_1^m\omega(x_1, \sigma),
$$
where $\omega$ is a smooth function and $m\ge2$ is a positive integer. Observe that if all derivatives of $ \psi(x_1+x_1(\sigma), \sigma)$ with respect to $x_1$ vanish at the origin, we may choose $m$ as large as we wish. Notice also that since $\psi(x_1,0)=0,$ we have that 
 $\om(x_1,0)=0.$

Finally, after applying the affine change of variables 
$$
(x_1,x_2-\psi(x_1(\sigma), \sigma)-\pa_1\psi(x_1(\sigma), \sigma)x_1)\mapsto (x_1,x_2 ),
$$
we arrive at the conclusion of the lemma. 
\end{proof}

\noi {\bf Proof of Theorem \ref{nondegairypar}.} By passing from the phase $\phi(x,\si)$ to $\tilde \phi(x,\si):=\phi(x,\si)-\phi(0,\si)-\nabla\phi(0,\si)\cdot x$ and applying a suitable linear change of coordinates to the ambient space $\RR^3,$ it is easily seen that we may assume without loss of generality that $\phi(0,\si)=0$ and $\nabla\phi(0,\si)=(0,0).$ 

 As a next  step, notice that the proof can be reduced to the case $\phi_r\equiv 0$ considered in Lemma \ref{simpar}, by adding   further parameters to $\si.$ 
 
 \smallskip
 Indeed, observe that we may write 
  $$
 \phi(x_1,x_2,\si)=\al_2(\si) x_2^2+\al_3(\si) x_1^3 +\phi_r(x,\si) +\beta_1(\si) x_1^2+\beta_2(\si)x_1x_2,
 $$
 where $\al_2(\si),\al_3(\si), \beta_1(\si)$ and $\beta_2(\si)$ are smooth functions of $\si$ such that $\al_2(0)=a_2, \al_3(0)=a_3$ and 
 $\beta_1(0)=\beta_2(0)=0,$ and where $\phi_r(x_1,x_2,\si)$ is smooth such that  $\N(\phi_r)\subset \{t_1/3+t_2/2>1\},$ for every $\si.$ 
 
Thus, if we put   $\phi_{(s)}(x_1,x_2,\si):=\frac 1s \phi( s^{1/3}x_1,s^{1/2}x_2,\si),\, s>0,$ then
$$
\phi_{(s)}(x_1,x_2,\si)=\al_2(\si) x_2^2+\al_3(\si) x_1^3 +s^{\frac 16} \phi_r(x,\si,s^{\frac 16}) +\frac {\beta_1(\si)}{s^{\frac 13}} x_1^2
+\frac{\beta_2(\si)}{s^{\frac 16}}x_1x_2,
 $$
 where also the new function $\phi_r$ is a smooth function of its arguments. Let us therefore view  $s_1=s^{ 1/6},$ 
 $s_2= {\beta_1(\si)}/{s^{1/3}}$ and $s_3=\beta_2(\si)/s^{1/6}$ as new, small parameters and  define a new parameter vector 
 $\tilde \si:=(s_1,s_2,s_3,\si).$ This allows to write  
$$
\phi_{(s)}(x_1,x_2,\si)=\al_2(\si) x_2^2+\al_3(\si) x_1^3 +s_1 \phi_r(x,\si,s_1) +s_2 x_1^2+s_3x_1x_2=:\tilde\phi(x_1,x_2,\tilde\si).
 $$
Notice here that if we first choose $s$ sufficiently small (the size of $s$ will  determine  how much we have shrunk  the support of the amplitude $\eta$),  and then $\si$ sufficiently small (depending on our chosen $s$), then we may indeed also assume that the new parameters are sufficiently small. This allows us to consider the components of $\tilde\si$ as independent, small parameters.

But then, $\tilde\phi(x_1,x_2,0)=a_2x_2^2+ a_3x_1^3,$ so that indeed $\tilde\phi_r\equiv 0.$ Moreover, our discussion shows that it  clearly  suffices to prove the theorem for the phase $\tilde\phi$ in place of $\phi$  and $\tilde \si$ sufficiently small.
 \smallskip
 
 Thus, let us henceforth assume that $\phi_r\equiv 0,$  so that  Lemma \ref{simpar} applies.
 Due to this lemma,  after applying a suitable  linear change of variables  (depending possibly on $\si$) to the ambient space $\RR^3,$ we may assume that  the maximal operator $\M^\si_T$ is of the following form:

\begin{equation}\label{maxnonaipar}
\M^\si_T f(y,y_3):=\sup_{t>0}\left|\int f(y-t(x+\al(\sigma)), y_3-t(T+E(\sigma)+\phi(x, \sigma)))\, a_0(x, \sigma) \,dx\right|,
\end{equation}
where $a_0$ is  again a smooth non-negative function supported in a
sufficiently small neighborhood of the origin and 
\begin{equation}\label{phase10}
\phi(x_1, x_2, \sigma)=b(x_1, x_2, \sigma)(x_2-x_1^m\omega(x_1, \sigma))^2+x_1^3\beta(x_1, \sigma), 
\end{equation}
and where $\al(\si)=(\al_1(\si),\al_2(\si))$ and $E(\si)$ are smooth functions of $\si$ such that 
$\al(0)=0$ and $1/2\le E(\si)\le 2,$ and where  $\om(x_1,0)=0.$

For the  proof of Theorem \ref{nondegairypar}, we shall thus assume that $\M^\si_T$ is given by \eqref{maxnonaipar}.  The proof will be  based on suitable dyadic  frequency space decompositions, also with respect to the distance to certain  ``Airy cones'' associated with our phase functions, and the study of the corresponding frequency-localized maximal operators.
For the sake of simplicity of notation, we shall usually suppress  the superscript  $\si$ in the proof.

\subsection{Dyadic decomposition with respect to the distance to the Airy cone}\label{airydecomp}

Note that the maximal operator $\M=\M^\si_T$ is associated to  the  averaging  operators given by  convolutions with dilates  of a measure $\mu$ whose Fourier transform at $\xi=(\xi_1,\xi_2,\xi_3)$  is given  by

\begin{equation}\label{fourtrans}
\hat\mu(\xi) :=\int e^{-i\big(\xi_1(x_1+\al_1(\sigma))+\xi_2(x_2+\al_2(\sigma))+\xi_3(T+E(\si)+\phi(x_1,x_2, \sigma))\big)} a_0(x_1,x_2, \sigma)\,dx_1 dx_2.
\end{equation}

In order to estimate this maximal operator, we shall perform a dyadic decomposition with respect to the last variable $\xi_3.$  The low frequency part is easily controlled by the Hardy-Littlewood maximal operator.   

 Let us  thus denote by $\la\ge2$  a sufficiently large dyadic number, and  decompose  $\hat \mu$  into  
\begin{equation}\nonumber
    \widehat{\mu^\la}(\xi) :=\chi_0\left(\frac{\xi_1}{\lambda},\,\frac{\xi_2}{\lambda} \right) \chi_1\left(\frac{\xi_3}{\lambda} \right)\hat\mu(\xi).
\end{equation}
and 
$$
\left(1-\chi_0\left(\frac{\xi_1}{\lambda},\,\frac{\xi_2}{\lambda} \right)\right)\chi_1\left(\frac{\xi_3}{\lambda} \right)\hat\mu(\xi),
$$
where $\chi_0$ and $\chi_1$ are again  smooth functions with sufficiently small compact supports, and $\chi_0$ is  identically $1$ on a small neighborhood of the origin, whereas  $\chi_1$ vanishes near the origin and is identically one  near $1.$

Notice that if we choose the support of $a_0$ sufficiently small, then integrations by parts easily show that the latter contribution is of order $O(\la^{-N})$ for every $N\in\NN$  as $\la\to+\infty, $ so that the corresponding contribution to the maximal operators is under control. 

It therefore suffices to control the contribution by $\mu^\la.$ 
Following  \cite{IMmon} we write
\begin{equation}\label{chv}
\xi_3=\la s_3,\quad \xi_1=\la s_3s_1,\quad \xi_2=\la s_3s_2,
\end{equation}
and put $s':=(s_1,s_2), s:=(s',s_3).$
Then we have 
$$
|s_3|\sim 1\quad \mbox{and}\quad  |s'|\ll1
$$ 
on the support of $\widehat{\mu^\la}.$  
Moreover,  writing 
$$
\xi_1x_1+\xi_2x_2+\xi_3\phi(x_1,\,x_2, \sigma)=: \la s_3\Phi(x,s', \sigma),
$$
where
$$
\Phi(x,s', \sigma):=s_1x_1+s_2x_2+\phi(x, \sigma),
$$
and putting 
$$\Gamma(\si):=(\al_1(\si),\al_2(\si),T+E(\si)),
$$ 
 we may re-write  
\begin{equation}\label{mulaj}
 \widehat{\mu^\la}(\xi) =e^{-i\xi\cdot\Gamma(\si)} \chi_0(s_3s')\chi_1(s_3) J(\la,s,\si),
\end{equation}
where $J(\la,s,\si)$ denotes the oscillatory integral
\begin{equation}\nonumber
J(\la,s,\si):=\int_{\bR^2} e^{-i\la s_3\Phi(x,s',  \sigma)} a_0(x, \sigma)\, dx.
\end{equation}
\smallskip

In view of \eqref{phase10}, we perform  the change of variables $x_2\mapsto x_2+x_1^m\omega(x_1, \sigma)$ and obtain
\begin{equation}\label{shosin}
J(\la,s,\si)=\int_{\bR^2} e^{-i\la s_3\Phi_1(x, s', \sigma)} a_0(x_1,\,x_2+x_1^m\omega(x_1, \sigma))\,dx,
\end{equation}
where
$$
\Phi_1(x, s', \sigma):=b(x_1,\,x_2+x_1^m\omega(x_1, \sigma),\si)x_2^2+s_2x_2+s_1x_1+s_2x_1^m\omega(x_1, \sigma)+x_1^3\beta(x_1, \sigma).
$$
Applying the method of  stationary phase to the integration in  $x_2,$ we next  obtain that
$$
J(\la,s,\si)=\la^{-1/2} \int_{\bR} e^{-i\la s_3\tilde \Psi(x_1,s',\sigma)} \tilde a(x_1, \sigma)\,dx_1+r(\la,s),
$$
where $\tilde a$ is another smooth bump function supported in a sufficiently small neighborhood of  the origin,  $r(\la,s)$ is a remainder term of order  
$$
r(\la,s)=O(\la^{-\frac32})\quad \mbox{as} \quad  \la\to+\infty,
$$ 
and the new phase $\tilde\Psi$ is given by  
$$
\tilde\Psi(x_1,s',\si):=\Phi_1(x_1,x_2^c(x_1, s_2, \sigma),s', \sigma),
$$
where $x_2^c(x_1, s_2, \sigma)$ denotes  the unique (non-degenerate) critical point  of the  phase $\Phi_1$ with respect to $x_2.$

The contribution of the error term $r(\la,s)$ to $\widehat{\mu^\la}$ and the corresponding maximal operator is easily estimated, and we shall therefore henceforth ignore it.

\medskip
In order to understand $\tilde\Psi,$ we need more information on the critical point $x_2^c(x_1, s_2, \sigma)$.
Note first that the critical point must be of the form $x_2^c(x_1, s_2, \sigma)=s_2w(x_1, s_2, \sigma),$ where, due to \eqref{eqnew1},  the function  $w$ is smooth and satisfies the condition $w(x_1,s_2,0)=-\frac1{2 a_2}\neq 0$.
Hence we have
$$
\tilde\Psi(x_1,s',\si)=s_2^2b_1(x_1, s_2, \sigma)+s_1x_1+s_2x_1^m\omega(x_1, \sigma)+x_1^3\beta(x_1, \sigma),
$$
 where $b_1$ is another smooth function, with  $b_1(x_1,s_2, 0)=-\frac1{4 a_2}\neq 0$.

Following  \cite{IMmon} we consider the equation
$$
\pa_{x_1}^2\tilde\Psi(x_1,s',\si)=0.
$$
    By the implicit function theorem,  we see in a similar way as before  that it has a solution  of the form $x_1^c(s_2, \sigma)=s_2G_1(s_2, \sigma)$, where $G_1$ is a smooth function such that $G_1(s_2,0)=0.$ 

   By performing finally the  translation  of the $x_1$-coordinate  $x_1\mapsto x_1+x_1^c(s_2, \sigma),$ we see that we may write (up to an error term which, as mentioned before,  can be ignored) 
   \begin{equation}\label{Jla2}
J(\la,s,\si)=\la^{-1/2} \int_{\bR} e^{-i\la s_3\Psi(x_1,s',\sigma)} a(x_1, s_2,\sigma)\,dx_1,
\end{equation}
with a smooth amplitude $a$ which has similar properties like $\tilde a,$ and 
where the new phase $\Psi$ is of the form
\begin{equation}\label{phpsi}
  \Psi(x_1,s', \sigma)=B_0(s',\sigma)+B_1(s', \sigma)x_1+x_1^3B_3(x_1,\,s_2, \sigma),
\end{equation}
where $B_0,\,B_1,\,B_3$ are smooth functions with $B_3(0, 0, 0)\neq0$. One easily checks that the functions $B_1,\,B_0$ have the forms

    \begin{equation}\label{Bjs}
  B_0(s', \sigma)=s_2^2\tilde
    b_1(s_2, \sigma)+s_1s_2G_1(s_2, \sigma),\quad
    B_1(s', \sigma)=s_1-s_2^{m_1}G_3(s_2, \sigma),
\end{equation}
where $\tilde b_1$ and $G_3$ are smooth functions with $\tilde
b_1(0,0)=b_1(0,0,0)\neq0$ and  $G_3(s_2,0)=0,$ and where $m_1$ is an integer  $m_1\ge2.$ 
    
Following  \cite{IMmon} (cf. Chapter 5), we next perform a dyadic frequency decomposition of $\mu^\la$ with respect to the ``Airy cone'' given by $B_1(s',\si)=0.$ 

 More precisely, we choose smooth cut-off functions $\chi_0$ and $\chi_1$ such that   $\chi_0=1$ on a sufficiently large neighborhood of the origin, and 
$\chi_1(t)$ is supported  where $|t|\sim 1$ and $\sum_{k\in \bZ} \chi_1(2^{-2k/3}t)=1$ on $\RR\setminus \{0\},$ and define the functions $\mu_{Ai}^\la$ and $\mu_{k}^\la$ by 
\begin{eqnarray*}  
\widehat{\mu_{Ai}^\la}(\xi)&:=&\chi_0\Big(\la^{\frac23}B_1(s',\si)\Big)\, \widehat{\mu^\la}(\xi),\\  
\index{nudelai@$\mu_{Ai}^\la$}    \index{nudell@$\mu_{\de,l}^\la$}
 \widehat{\mu_{k}^\la}(\xi)&:=&\chi_1\Big((2^{-k}\la)^{\frac23}B_1(s',\si)\Big)\, \widehat{\mu^\la}(\xi), \qquad M_0\le  2^k\le \frac \la {M_1},  
\end{eqnarray*}
so that 
\begin{equation}\label{ai5+}
\mu^\la= \mu_{Ai}^\la+\sum_{M_0\le  2^k\le \frac \la {M_1}}\mu_{k}^\la.
\end{equation}
Here, we may assume that $M_0$ and $M_1$ are sufficiently large positive numbers.
 We shall separately   estimate the contributions $\M_{Ai}^\la$ by $\mu_{Ai}^\la$   and $\M_k^\la$  by  the  $\mu_{k}^\la$ to our maximal operator.   More precisely, we shall prove the following lemma:
 \begin{lemma}\label{Mlaest}
 Let $1<p\le 2.$ Then for every  $\de >0,$ we have 
 \begin{equation}\label{maxaies1}
\|\M_{Ai}^\la \|_{L^p\mapsto L^p}
 \le C_{p,\de}\,T^{\frac 1p+\de} \la^{2(\frac 1p -\frac 23) +\de(\frac 2p -1)}
\end{equation}
and 
\begin{equation}\label{maxkest1}
\|\M_{k}^\la \|_{L^p\mapsto L^p}
 \le C_{p,\de} \,T^{\frac 1p+\de} 2^{k(\frac 1p -\frac 23) -\de(\frac 2p -1)}\,\la^{2(\frac 1p -\frac 23) +\de(\frac 2p -1)},
\end{equation}
where the constant $C_{p,\de}$ does not depend on $\si,T $ and $\la.$ 
\end{lemma}
If $p>3/2,$ we see that these estimates sum in $k$ as well as over all dyadic $\la\gg1,$ provided we choose $\de$ sufficiently small. This will then complete the proof of Theorem \ref{nondegairypar}.

\subsection{The contribution by the region near the Airy cone}\label{nearairy}
In this subsection, we shall prove estimate \eqref{maxaies1}.
To this end, we write, according to \eqref{mulaj},  
$$
 \widehat{\mu^\la_{Ai}}(\xi) =e^{-i\xi\cdot\Gamma(\si)}  \chi_0\Big(\la^{\frac23}B_1(s',\si)\Big)\, \chi_0(s_3s')\chi_1(s_3) J(\la,s).
 $$
Then,  by \eqref{Jla2}, we may apply  Lemma 2.2 (a) (with $B=3$) in Chapter 2  of \cite{IMmon} to $J(\la,s)$ and see that we may write
$$
 \widehat{\mu^\la_{Ai}}(\xi)=e^{-i\xi\cdot\Gamma(\si)} \, \la^{-\frac 56} \chi_0\Big(\la^{\frac23}B_1(s',\si)\Big)\, g(\la^\frac23B_1(s',\si), \la,  \sigma,s)\chi_0(s_3s')\chi_1(s_3)\, e^{-i\la s_3B_0(s', \sigma)},
$$
where $g(u, \la,  \sigma,s)$ is a smooth function of $(u, \la,  \sigma,s)$ whose derivatives of any order are uniformly bounded on its natural domain $|u|\lesssim1, |\si|\lesssim1,\la\ge2$ and 
$|s_3|\sim 1, \,  |s'|\ll1.$ 

\smallskip

Notice also that a first order $t$-derivative of $e^{-i t \xi\cdot\Gamma(\si)},$  as well as of  
$e^{-i\la ts_3B_0(s', \sigma)},$ produces additional factors of order $O(T\la),$ since 
$|\Gamma(\si)|\sim T,$ and thus following again the arguments in the proof of Proposition \ref{simple} we easily see that 
\begin{equation}\label{maxai}
\|\M^\la_{Ai}\|_{L^2\to L^2}\lesssim T^{\frac 12}\la^{\frac 12-\frac56}= T^{\frac 12}\la^{-\frac13}.
\end{equation}

Next, we we shall control  the  function $\mu^\la_{Ai}(y).$ By Fourier inversion,  changing variables as in \eqref{chv}, we may write
 \begin{eqnarray}\nonumber
\mu^\la_{Ai}(y+\Gamma(\si))&=&\la^\frac{13}6\int \chi_0\Big(\la^{\frac23}B_1(s',\si)\Big)
g(\la^\frac23B_1(s',\si), \la,  \sigma,s)\chi_0(s_3s')\chi_1(s_3) \\
&&\hskip5cm e^{-i\la s_3 \big(B_0(s', \sigma)-s_1y_1-s_2y_2-y_3\big)}s_3^2\, ds. \label{measai}
\end{eqnarray}

Observe first that when $|y|\gg 1,$ then we easily obtain by means of integrations by parts that
\begin{equation}\nonumber 
|\mu_{Ai}^\la(y)|\le C_N \la^{-N}, \quad N\in\NN, \mbox { if } |y|\gg 1.
\end{equation}
Indeed, when $|y_1|\gg 1,$ then we integrate by parts repeatedly in $s_1$ to see this (in each integration by parts, we gain a factor $\la^{-1}$ and loose at most a factor of $\la^{2/3}$),  and a similar argument applies when $|y_2|\gg1,$ where we use the $s_2$-integration.  Finally,  when $|y_1|+|y_2|\lesssim 1$ and $|y_3|\gg 1,$ then we can integrate by parts in $s_3$ in order to establish this estimate.
\medskip

We may therefore assume in the sequel  that $|y|\lesssim 1.$ 
In view of \eqref{Bjs}, we  then perform yet another change of coordinates from $s_1$ to 
$$z:=\la^\frac23B_1(s',\si)=\la^\frac23(s_1-s_2^{m_1}G_3(s_2, \sigma)),\mbox{ i.e.,} \ 
s_1=s_2^{m_1}G_3(s_2, \sigma)+\lambda^{-\frac{2}{3}}z.
$$
Note that  by \eqref{Bjs},  the function $B_0(s', \si)$ is  then given by 
$$
B_0(s', \si)=s_2^2G_5(s_2, \si)+s_2G_1(s_2, \si)\la^{-\frac 23}z,
$$
where $G_5(s_2,\si):=\tilde b_1(s_2, \si)+s_2^{m_1-1}G_1(s_2, \si)G_3(s_2, \si)$ is smooth and  where we may assume that $G_5(0,0)\ne 0,$ since  $m_1\ge2$ and $G_1(s_2,0)=0.$ 
We may thus re-write 
\begin{eqnarray}\nonumber
\mu_{A_i}^\lambda(y+\Ga(\si))&=&\lambda^\frac32\int \chi_0(z) \tilde g(z, \la,  \sigma,s_2,s_3)\chi_0(z,s_2,s_3)\chi_1(s_3)\\
&& \hskip4cm e^{-i\lambda s_3\Phi_2(z, s_2, y, \sigma)}ds_2ds_3dz, \label{muAi3}
\end{eqnarray}

where $\tilde g$  has similar properties to $g,$ and where the phase $\Phi_2$ is of the form
\begin{eqnarray*}
\Phi_2(z, s_2, y,\si)&=&s_2^2G_5(s_2, \sigma)-s_2^{m_1}G_3(s_2, \sigma)y_1 -s_2y_2-y_3 \\
&&\hskip4cm +\lambda^{-\frac{2}{3}}z(s_2G_1(s_2, \sigma)-y_1),
\end{eqnarray*}
with  $G_5(0,0)\ne 0$ and  $G_1(s_2,0)=G_3(s_2,0)=0$ (compare this with the analogous formula (5.21) in \cite{IMmon}). Recall also   that $|z|\lesssim 1, |s_2|\ll 1$ and $|s_3|\sim 1.$

Decomposing 
\begin{equation}\label{G1decomp}
G_1(s_2,\si)=G_{1,1}(\si)+s_2 G_{1,2}(s_2, \si),
\end{equation}

where $G_{1,1}(0)=0$ and $G_{1,2}(s_2,0)=0,$ 
we finally  may re-write 
\begin{eqnarray}\nonumber
\Phi_2(z, s_2, y,\si)&=&s_2^2G_6(\lambda^{-\frac{2}{3}}z,s_2, \sigma)
-s_2\big(y_2-\lambda^{-\frac{2}{3}}z\,G_{1,1}(\sigma)\big)\\
&&\hskip3cm  -s_2^{m_1}G_3(s_2, \sigma)y_1  -\lambda^{-\frac{2}{3}}zy_1 -y_3  , \label{Phi1}
\end{eqnarray}
where $G_6(0,0,0)\ne 0.$

\smallskip
Assuming $\si$ to be sufficiently small, we thus see that $\Phi_2$ has a unique critical point $s_2^c$ with respect to $s_2,$ of size $|y_2|,$ provided $|y_2|\ll 1.$ Otherwise, iterated integrations by parts in $s_2$ will show that $|\mu_{A_i}^\lambda(y)|\lesssim \la^{-N}$ for every $N\in\NN,$ so that those $y's$ can be ignored. More precisely, from \eqref{Phi1}  we deduce that $s_2^c$ is of the form
$$
s_2^c=\big(y_2-\lambda^{-\frac{2}{3}}z\,G_{1,1}(\sigma)\big) H_1(\lambda^{-\frac23}z, y_1, y_2, \si).
$$
After applying the method of  stationary phase in the $s_2$-variable, we thus find that
$$
\mu^\la_{Ai}(y+\Ga(\si))=\la\int e^{-i\la s_3
\Phi_3(\lambda^{-\frac23}z, y, \si)} \chi_0(z,s_3)\chi_1(s_3)a(\la^{-\frac23} z, y,  s_3,\sigma)\, dzds_3+O(1),
$$
with a smooth amplitude $a,$ and where the phase  $\Phi_3$ is of the form
$$
\Phi_3(\lambda^{-\frac23}z, y, \si)=\big(y_2-\lambda^{-\frac{2}{3}}z\,G_{1,1}(\sigma)\big)^2  H_2(\lambda^{-\frac23}z, y_1, y_2, \si)-\lambda^{-\frac23}zy_1-y_3,
$$
with a smooth function $H_2$ such that $H_2(0,0,0,0)\ne 0.$ Notice that the functions $\chi_0$ and $\chi_1$ may possibly be different in different places.

\bigskip

A Taylor series expansion of $H_2$  with respect to $v=\lambda^{-\frac23}z$ then allows to write
\begin{eqnarray}\nonumber
\Phi_3(\lambda^{-\frac23}z, y, \si)&=&y_2^2H_2(0, y_1, y_2, \si)+
\la^{-\frac{2}{3}}z\big(-2y_2G_{1,1}(\si)H_2(0, y_1, y_2, \si)\\
&+&y_2^2 \pa_vH_2(0, y_1, y_2, \si)-y_1\big)-y_3+O(\la^{-4/3}).\label{Phi3Ai2}
\end{eqnarray}
The factor $e^{-i\la s_3O(\la^{-4/3})}$ corresponding to the term $O(\la^{-4/3})$ can be included into the amplitude, and thus we may assume that the complete phase is  of the form
$$
\la s_3\Phi_3=\lambda^{\frac13}s_3z(-y_2H_3(y_1, y_2, \si)+y_2^2H_4(y_1, y_2, \si)-y_1)
+\la
s_3(y_2^2H_2(0, y_1, y_2, \si)-y_3).
$$
By the implicit function theorem, we may re-write the first factor in parentheses in the form
$$
-y_2H_3(y_1, y_2, \si)+y_2^2H_4(y_1, y_2, \si)-y_1=(y_1-\vp(y_2, \si))H_5(y_1, y_2, \si),
$$
where $\vp(y_2, \si)$ and   $ H_5(y_1, y_2, \si)$ are smooth functions with
$\varphi(0, 0)=0$ and $H_5(0, 0, 0)\neq 0$. We shall also write $\psi(y_1,y_2,\si):=
y_2^2H_2(0, y_1, y_2, \si).$ Then also $\psi$ is smooth and $\psi(0,0,0)=0.$

Thus, eventually we may write 
$$\mu^\la_{Ai}(y+\Ga(\si))=\mu^\la_{I}(y+\Ga(\si))+\mu^\la_{II}(y+\Ga(\si)),$$
 where $\mu^\la_{II}(y+\Ga(\si))=O(1),$ and 
\begin{equation}\label{Phi3AI}
\mu^\la_{I}(y+\Ga(\si)):=\la\int e^{-i\la s_3
\Phi_3(\lambda^{-\frac23}z, y, \si)} \chi_0(z,s_3)\chi_1(s_3)a(\la^{-\frac23} z, y,  s_3,\sigma)\, dzds_3,
\end{equation}
with 

$$
\la \Phi_3(\lambda^{-\frac23}z, y, \si)=\lambda^{\frac13}z\,\big(y_1-\vp(y_2, \si))H_5(y_1, y_2, \si)
-\la \big(y_3-\psi(y_1,y_2,\si)\big).
$$

\smallskip

The contribution of $\mu^\la_{II}$ to the maximal operator is of order $O(T),$ as can easily be seen by comparison  with the Hardy-Littlewood maximal operator (or by Poposition \ref{maxproj}).

As for $\mu^\la_{I},$ fix a sufficiently small number $\de>0.$ If  $\la^\frac13|y_1-\varphi(y_2, \si)|\ge c\la^\delta$, where $c>0$ is assumed to be a sufficiently small constant, then we can use iterated integrations by parts in $z$ to  obtain that $|\mu^\la_{Ai}(y+\Ga(\si))|\lesssim \la^{-N}$  for every $N,$ uniformly in $\si.$  Similarly, if $\la^\frac13|y_1-\varphi(y_2, \si)|< c\la^\delta$ and $\la|y_3-\psi(y_1,y_2,\si)|\ge \la^\delta$, then we can use integrations by parts in $s_3$ to arrive at the same type of estimate. 
\color{black}

The contributions by these regions to our maximal operator are thus  negligible.

Let us next put 
\begin{eqnarray*}
B_\de:=\{y\in\RR^3: |y|\lesssim 1,\,  |y_1-\varphi(y_2, \si)|< \la^{\delta-\frac 13}, \, 
 |y_3-\psi(y_1,y_2,\si)|<\la^{\de-1}\}.
\end{eqnarray*}
Writing $x=y+\Ga(\si),$ we are  thus  reduced to concentrating for $\mu^\la_{Ai}(x)$ on the  small $x$- region $A_\de:=\Ga(\si)+B_\de,$ on which we only have the estimate $|\mu^\la_{Ai}(x)|\lesssim \la.$ 
But recall  that $\Ga(\si)=(\al_1(\si),\al_2(\si),T+E(\si)),$ with $|\al_1(\si)|,|\al_2(\si)| \ll 1,$ $1/2\le E(\si)\le 2$  and $T\gg1.$
 It is then  obvious that the spherical projection $\pi(A_\de)\subset S^2$ of $A_\de$ has measure $|\pi(A_\de)|\lesssim T^{-2}\la^{\de-\frac13},$ since $\la^{\de-1}\ll\la^{\delta-\frac 13},$ 
so that Proposition \ref{maxproj} (with $R\sim T$)  shows that,  for every $\ve>0,$ $\M_{Ai}^\la$ satisfies the  estimate

\begin{equation}\label{maxaione}
\|\M_{Ai}^\la \|_{L^{1+\ve}\mapsto L^{1+\ve}}\le C'_{\ve,\de}T^{3}T^{-2}\la \cdot \la^{\de-\frac 13}+ O(T) \le  C_{\ve,\de} T\la^{\frac23+\de},
\end{equation}
for every  $\de>0.$

The estimate \eqref{maxaies1} in Lemma \ref{Mlaest} now follows from  the estimates \eqref{maxai}
 and \eqref{maxaione} by real interpolation (with a slightly bigger $\de$ than the one considered here), if we choose $\ve$ sufficiently small.

\subsection{The contribution by the region away from  the Airy cone}\label{awayairy}
According to \eqref{mulaj},  we have 
$$
 \widehat{\mu^\la_{k}}(\xi) =e^{-i\xi\cdot\Gamma(\si)} \chi_1\Big((2^{-k}\la)^{\frac23}B_1(s',\si)\Big)\,  \chi_0(s_3s')\chi_1(s_3) J(\la,s).
 $$
By \eqref{Jla2} (ignoring again the error term $r(\la,s)$), this can be re-written as

\begin{equation}\label{measail}
\widehat{\mu^\la_{k}}(\xi) =\la^{-\frac12} e^{-i\xi\cdot\Gamma(\si)} \chi_1\Big((2^{-k}\la)^{\frac23}B_1(s',\si)\Big)\,  \chi_0(s_3s')\chi_1(s_3)e^{-i\la
s_3B_0(s', \si)} \tilde J(\la,s,\si),
\end{equation}
where
$$
\tilde J(\la, s,\si):=\int_{\bR} e^{-i\la s_3\tilde \Psi(x_1,s',\sigma)} a(x_1, s_2,\sigma)\,dx_1,
$$
with phase function
$$
\tilde \Psi(x_1,\,s', \si):=B_1(s', \sigma)x_1+x_1^3B_3(x_1,\,s_2, \sigma).
$$
Recall also that the amplitude $a$ is smooth and has a sufficiently small support in $x_1.$ 

Since $B_1(s',\si)$ is of size $(2^{k}\la^{-1})^\frac23$, we  scale by changing coordinates $x_1= (2^k\la^{-1})^\frac13u_1$ in the integral for $\tilde J(\la, s,\si)$ and obtain
$$
\tilde J(\la,s,\si)= (2^k\la^{-1})^\frac13\int e^{-i2^k s_3
\Psi_1(u_1,\,s', \si)}a((2^k\la^{-1})^\frac13u_1,s_2,\si)\,du_1,
$$
where
\begin{equation}\label{Psi1}
\Psi_1(u_1, s', \si):=(2^{-k}\la)^\frac23B_1(s', \si)u_1+B_3( (2^k\la^{-1})^\frac13u_1,s_2, \si)u_1^3.
\end{equation}

Observe that the coefficients of $u_1$ and $u_1^3$ in $\Psi_1$ are
of size $1$, so that  $\Psi_1$ has no critical point with respect to
$u_1$ unless $|u_1|\sim1$. Thus we may choose a smooth cut-off
function $\chi_1\in C_0^\infty(\bR)$ which vanishes near the origin so that $\Psi_1$ has no critical point on the support of the
function $1-\chi_1,$ and decompose the integral 
 $$
 \tilde J(\la,s,\si)=J_1(\la, s,\si)+J_\infty(\la, s,\si),
 $$
where
\begin{equation}\label{J1}
J_1(\la, s,\si):= (2^k\la^{-1})^\frac13\int e^{-i2^k s_3
\Psi_1(u_1,\,s', \si)}a((2^k\la^{-1})^\frac13u_1,s_2,\si) \chi_1(u_1)\,du_1
\end{equation}
and 
\begin{equation}\label{Jinfty}
J_\infty(\la, s,\si):= (2^k\la^{-1})^\frac13\int e^{-i2^k s_3
\Psi_1(u_1,\,s', \si)}a((2^k\la^{-1})^\frac13u_1,s_2,\si)(1- \chi_1(u_1))\,du_1.
\end{equation}
Accordingly we decompose the measure
$\mu_k^\la=\mu_{k, 1}^\la+\mu_{k, \infty}^\la$, where the summands
are given  by
$$
\widehat{\mu_{k,1}^\la}(\xi)=\la^{-\frac12} e^{-i\xi\cdot\Gamma(\si)} \chi_1\Big((2^{-k}\la)^{\frac23}B_1(s',\si)\Big)\,  \chi_0(s_3s')\chi_1(s_3)e^{i\la
s_3B_0(s', \si)} J_1(\la,s,\si).$$
and
$$
\widehat{\mu_{k, \infty}^\la}(\xi)=\la^{-\frac12} e^{-i\xi\cdot\Gamma(\si)} \chi_1\Big((2^{-k}\la)^{\frac23}B_1(s',\si)\Big)\,  \chi_0(s_3s')\chi_1(s_3)e^{i\la
s_3B_0(s', \si)} J_\infty(\la, s, \si).
$$
We denote by $ \M_{k,1}^\la$ and $\M_{k, \infty}^\la$ the maximal operators defined by the functions $\mu_{k,1}^\la$ and $\mu_{k,\infty}^\la,$ respectively.

\subsubsection{The  contributions given by the $\mu_{k, \infty}^\la$ }

Recall that $2^k\la^{-1}\ll 1.$ By means of integrations by parts, we then easily see  that, given any  $N\in \bN,$  we may write 
\begin{equation}\label{Jlainf1}
J_\infty(\la, s, \si)=(2^k\la^{-1})^\frac13 2^{-kN}  g_N\big((2^{-k}\la)^\frac23B_1(s', \si),2^{k}\la^{-1}, s',\si\big),
\end{equation}
where $g_N$ is smooth. This implies in particular that 
$$
|J_\infty(\la, s, \si)|\lesssim(2^k\la^{-1})^\frac13 2^{-kN},
$$
hence
$$
\|\widehat{\mu_{k, \infty}^\la}\|_\infty\lesssim
\la^{-\frac12}(2^k\la^{-1})^\frac13 2^{-kN}.
$$
Arguing as before, we thus find that 
\begin{equation}\label{l2estl}
\|\M_{k, \infty}^\la\|_{L^2\mapsto
L^2}\lesssim T^{\frac 12}(2^k\la^{-1})^\frac13 2^{-kN}
\end{equation}
for every $N\in\NN.$ 

Next, we proceed in a similar way as in the previous subsection in order to obtain $L^p$-estimates for $\M_{k, \infty}^\la$ when  $p$ is close to $1.$  By Fourier inversion and \eqref{Jlainf1}, we find that 

\begin{eqnarray*}
&&\mu_{k,\infty}^\la(y+\Ga(\si))=\la^3\int_{\bR^3} e^{i\la s_3
(s_1y_1+s_2y_2+y_3)}e^{i\xi\cdot\Gamma(\si)}\widehat{\mu_{k,\infty}^\la} (\xi)ds\\
&& \hskip-0,5cm =\la^{\frac52} \int  \chi_1\Big((2^{-k}\la)^{\frac23}B_1(s',\si)\Big)\,  \chi_0(s_3s')\chi_1(s_3)
e^{-i\la s_3\big (B_0(s', \si) -s_1y_1-s_2y_2-y_3\big)} J_\infty(\la, s, \si)\, ds,\\
&&  \hskip-0,5cm =\la^{\frac52}(2^k\la^{-1})^\frac13 2^{-kN}
 \int e^{-i\la s_3\big (B_0(s', \si) -s_1y_1-s_2y_2-y_3\big)}
\chi_1\Big((2^{-k}\la)^{\frac23}B_1(s',\si)\Big)\,  \chi_0(s_3s')\chi_1(s_3)\\
&&\hskip8cm g_N\big((2^{-k}\la)^\frac23B_1(s', \si),2^{k}\la^{-1}, s',\si\big)\, ds,
\end{eqnarray*}
where we recall that  $\xi=\la s_3(s_1, s_2, 1)$.  

Comparing this with \eqref{measai} and arguing in a similar way is we did for the estimation of $\mu_{Ai}^\la,$ we again see that we may assume that $|y|\lesssim 1,$ since for $|y|\gg 1$ integrations by parts in $s$ show that, for every $N\in \NN,$ 
\begin{equation}\nonumber 
|\mu_{k,\infty}^\la(y)|\le C_N 2^{-kN} \la^{-N}, \quad N\in\NN, \mbox { if } |y|\gg 1.
\end{equation}



Recalling \eqref{Bjs}, we next change coordinates from $s_1$ to 
$$
z:=(2^{-k}\la)^{\frac23}B_1(s_1,s_2,\si)=(2^{-k}\la)^{\frac23}(s_1-s_2^{m_1}G_3(s_2, \sigma)),
$$
i.e., 
\begin{equation}\label{changez}
s_1=s_2^{m_1}G_3(s_2, \sigma)+(2^{k}\la^{-1})^{\frac23}z, \qquad |z|\sim 1.
\end{equation}

Note that   the function $B_0(s', \si)$ is  then given by 
$$
B_0(s', \si)=s_2^2G_5(s_2, \si)+s_2G_1(s_2, \si)(2^k\la^{-1})^\frac23z,
$$
where $G_5(s_2,\si):=\tilde b_1(s_2, \si)+s_2^{m_1-1}G_1(s_2, \si)G_3(s_2, \si)$ is smooth and  where we may assume that $G_5(0,0)\ne 0,$ since  $m_1\ge2$ and $G_1(s_2,0)=0.$  We then  find that, in analogy with \eqref{muAi3}, 
\begin{eqnarray}\nonumber
&&\mu_{k,\infty}^\la(y+\Ga(\si))
 =\la^{\frac52}(2^k\la^{-1}) 2^{-kN}
 \int e^{-i\la s_3\Phi_2(z,s_2,y,\si)}
\chi_1(z) \,  \chi_0(s_3s')\chi_1(s_3)\\
&&\hskip8cm a_N\big(z,(2^{k}\la^{-1})^{\frac 13}, s_2,s_3,\si\big)\, ds_2 ds_3 d z,\label{muAk2}
\end{eqnarray}
where $a_N$ is smooth, and where the phase function $\Phi_2$ is given by 
\begin{eqnarray}\nonumber
\Phi_2(z, s_2, y,\si)&=&s_2^2G_5(s_2, \sigma)-s_2^{m_1}G_3(s_2, \sigma)y_1 -s_2y_2-y_3 \\
&&\hskip4cm +(2^k\lambda^{-1})^{\frac{2}{3}}z(s_2G_1(s_2, \sigma)-y_1).\label{Phi22}
\end{eqnarray}
Here, $G_5(0,0)\ne 0, G_1(s_2,0)=G_3(s_2,0)= 0,$ $2^k\la^{-1} \ll1,$  $|z|\sim 1, |s_2|\ll1$ 
and $|s_3|\sim 1.$ 
 \medskip
 
 Observe that, due to the presence of the factor $2^{-kN}$ in \eqref{muAk2}, we may concentrate primarily on gaining negative powers of $\la$ in the estimation of $\mu_{k,\,\infty}^\la.$ We may thus proceed very much in the same way as we did in the previous section,  replacing $v=\la^{-2/3}z$ by $v:= (2^k\la^{-1})^{2/3}z,$ and arrive at the following analogue of \eqref{Phi3AI}:
 
\begin{eqnarray}\nonumber
\mu^\la_{k,\infty}(y+\Ga(\si))&=&2^{-kN}\la\int e^{-i\la s_3
\Phi_3((2^k\la^{-1})^{\frac 23}z, y, \si)}\\
&&\hskip2cm  a_N(z,(2^k\la^{-1})^{\frac 23},s_3,y, \si)\chi_1(z)\chi_1(s_3)\, dzds_3+O(2^{-kN}), \label{Phi3k2}
\end{eqnarray}
where here $a_N$  is of the form
\begin{eqnarray*}
 a_N(z,(2^k\la^{-1})^{\frac 23},s_3,y, \si)=\tilde a_N(z,(2^k\la^{-1})^{\frac 23}, s_3,y,\sigma)\,
e^{-i\la s_3(2^k\la^{-1})^{\frac 43}\eta\big((2^k\la^{-1})^{\frac 23}z, y, \si\big)},
\end{eqnarray*}
 with a smooth function $\tilde a_N$ and a smooth,  real function  $\eta,$   and where   $\la\Phi_3$ is  of the form

$$
\la\Phi_3((2^k\la^{-1})^{\frac 23}z, y, \si)=\lambda (2^k\la^{-1})^{\frac 23}z\,\big(y_1-\vp(y_2, \si))H_5(y_1, y_2, \si)\big)
-\la \big(y_3-\psi(y_1,y_2,\si)\big),
$$

with $H_5(0,0,0)\ne 0.$ 
The oscillatory factor $e^{-i\la s_3(2^k\la^{-1})^{\frac 43}\eta((2^k\la^{-1})^{\frac 23}z, y, \si)}$ in this amplitude corresponds to the $O(\la^{-4/3})$ - term in \eqref{Phi3Ai2}. The contribution by the $O(2^{-kN})$ term is negligible, in a similar way as we saw  this for the contribution of the term $\mu^\la_{II}$ before, so that we shall from now on ignore it.

Fix again a  sufficiently small number $\de>0.$ If  $2^{\frac {2k}3}\la^\frac13|y_1-\varphi(y_2, \si)|\ge c\la^\delta$, where $c>0$ is assumed to be a sufficiently small constant,  then we can use iterated integrations by parts in $z$ to  obtain that $|\mu^\la_{k,\infty}(y+\Ga(\si))|\lesssim (2^k\la)^{-N}$  for every $N,$ uniformly in 
$\si.$ Indeed, since $\la (2^k\la^{-1})^{6/3}=2^{2k}\la^{-1}\lesssim 2^k,$  in each integration by parts, we gain a factor $\la^{-\de}$ and loose a factor $2^k$ from differentiating the amplitude, but this is acceptable. Similarly, if  $2^{\frac {2k}3}\la^\frac13|y_1-\varphi(y_2, \si)|< c\la^\delta$ and  $\la|y_3-\psi(y_1,y_2,\si)|\ge \la^\delta$, then we can use integrations by parts in $s_3$ to arrive at the same type of estimate.

The contributions by these regions to our maximal operator are thus  negligible.
\smallskip

Let us next put 
\begin{eqnarray*}
B_{k,\de}:=\{y\in\RR^3: |y|\lesssim 1,\,  |y_1-\varphi(y_2, \si)|< 2^{-\frac {2k}3}\la^{\delta-\frac 13},\,  |y_3-\psi(y_1,y_2,\si)|<\la^{\de-1}\}.
\end{eqnarray*}
As in the preceding subsection, let us write $x=y+\Ga(\si).$ We are  then  reduced to concentrating for $\mu^\la_{k,\infty}(x)$ on the  small $x$- region $A_{k,\de}:=\Ga(\si)+B_{k,\de},$ on which we  have the estimate $|\mu^\la_{k,\infty}(x)|\lesssim 2^{-kN}\la.$
Recalling again  that $\Ga(\si)=(\al_1(\si),\al_2(\si),T+E(\si)),$ with $|\al_1(\si)|,|\al_2(\si)| \ll 1,$ $1/2\le E(\si)\le 2$  and $T\gg1,$ we then see that the spherical projection  $\pi(A_{k,\de})\subset S^2$ of $A_{k,\de}$ has measure $|\pi(A_{k,\de})|\lesssim T^{-2} 2^{-\frac {2k}3}\la^{\de-\frac13},$ 
so that Proposition  \ref{maxproj} shows that,  for every $\ve>0,$ $\M_{k,\infty}^\la$ satisfies the  estimates

\begin{equation}\label{maxkinftyone}
\|\M_{k,\infty}^\la \|_{L^{1+\ve}\mapsto L^{1+\ve}}\le C_{N,\ve}T2^{-kN}\la \cdot \la^{\de-\frac 13} =T 2^{-kN}\la^{\frac23+\de},
\end{equation}
for every  $\de>0,$ uniformly in $\si.$

\subsubsection{The  contributions given by the $\mu_{k, 1}^\la$ }
We next  estimate the maximal operators $\M_{k,1}^\la$.
First, by applying the method of stationary phase to the integral $J_1(\la,s),$  we easily see that 
$$
\|\widehat{\mu_{k,1}^\la}(\xi)\|_\infty \lesssim  \la^{-\frac 12} (2^k\la^{-1})^{\frac 13} 2^{-\frac k2}=2^{-\frac k6} \la^{-\frac 56}.
$$
This implies that 
\begin{equation}\label{maxlpr}
\|\M_{k,1}^\la\|_{L^2\mapsto L^2}\lesssim T^{\frac 12}\la^{\frac 12} 2^{-\frac k6} \la^{-\frac 56}= T^{\frac 12}2^{-\frac k6} \la^{-\frac 13}.
\end{equation}

\medskip
In order to obtain $L^p$-estimates for $\M_{k, \infty}^\la$ when  $p$ is close to $1,$  we have to analyze more carefully which phase function arises  from the application of the method of stationary phase to $J_1(\la,s).$ To this end, let us put again $z:=(2^{-k}\la)^{\frac23}B_1(s',\si),$  so that the phase \eqref{Psi1} in $J_1(\la,s)$   can be written     
$$
\Psi_1(u_1, s_2,z, \si)=zu_1+B_3( (2^k\la^{-1})^\frac13u_1,s_2, \si)u_1^3,
$$
where $|u_1|\sim 1.$ We may assume that  it has a critical  point 
$u_1^c=u_1^c(s_2,z,s_2)$  of size $|u_1^c|\sim 1,$ for other wise we can integrate by parts in $u_1$ and can then proceed as we did for the $\mu_{k,\infty}^\la.$ Assuming for instance that $z>0,$  by writing $u_1=z^{1/2} w_1,$ we see that $u_1^c$ is of the form 
$u_1^c=z^{1/2} W((2^k\la^{-1})^\frac13 z^{\frac 12},s_2,\si),$ where $W$ is smooth and $|W|\sim 1.$ Moreover, since $B_3(\cdot,0,0)\ne 0$ is constant,  it is easy to see that  by choosing $\si$ and $|s'|$ sufficiently small, we get 
$$
\Psi_1(u_1^c, s_2,z, \si)=z^{\frac 32}H_1((2^k\la^{-1})^\frac13z^{\frac 12},s_2,\si),
$$
with a smooth function $H_1$ such that $|H_1|\sim 1.$ We may thus write 
\begin{equation}\label{J12}
J_1(\la, s,\si)= 2^{-\frac k 6}\la^{-\frac 13}\, e^{-i2^k s_3
z^{\frac 32}H_1((2^k\la^{-1})^\frac13z^{\frac 12},s_2,\si)}a((2^k\la^{-1})^\frac13z^{\frac 12},s_2,\si)s_3^{-\frac 12},
\end{equation}
again with a smooth amplitude $a.$ More precisely, we would also pick up an error term of order $O(2^{-\frac {7k }6}\la^{-\frac 13}),$ which we shall here ignore for the sake of simplicity of the presentation. As explained in \cite{IMmon}, such an error term could be avoided by replacing the ``gain'' $2^{-k/2}$ in this application of the method of stationary phase by a symbol of order $-1/2$ in $2^k$ (which would, however, also depend on further  
variables). Changing then again variables from $s_1$ to $z,$ in analogy with \eqref{muAk2} and \eqref{Phi22} we thus find that 

\begin{eqnarray}\nonumber
&&\mu_{k,1}^\la(y+\Ga(\si))
 =\la^{\frac52}(2^k\la^{-1}) 2^{-\frac k2}
 \int e^{-i \la s_3\Phi_2(z,s_2,y,\si)}
\chi_1(z) \,  \chi_0(s_3s')\chi_1(s_3)\\
&&\hskip8cm a\big(z,(2^{k}\la^{-1})^{\frac 13}, s_2,\si\big)\, ds_2 ds_3 d z,\label{muAk3}
\end{eqnarray}
where $a$ is smooth, and where the phase function $\Phi_2$ is given by 
\begin{eqnarray*}\nonumber
\Phi_2(z, s_2, y,\si)&=&2^k \la^{-1}
z^{\frac 32}H_1((2^k\la^{-1})^\frac13 z^{\frac 12},s_2,\si)\\
&&\hskip-1cm + 
s_2^2G_5(s_2, \sigma)-s_2^{m_1}G_3(s_2, \sigma)y_1 
 -s_2y_2-y_3+(2^k\lambda^{-1})^{\frac{2}{3}}z(s_2G_1(s_2, \sigma)-y_1).
\end{eqnarray*}
Here, $G_5(0,0)\ne 0, G_1(s_2,0)=G_3(s_2,0)= 0,$ $2^k\la^{-1} \ll1,$  $|z|\sim 1, |s_2|\ll1$ 
and $|s_3|\sim 1.$ 

In particular, we see that we  can next apply the method of stationary phase to the integration in $s_2$ (assuming that there is a critical point $s_2^c$ within the domain of integration; otherwise,  we may pick up powers $\la^{-N}$ by means of integrations by parts and are done). 
Let us also write 
$$q:= (2^k\la^{-1})^\frac13,$$
so that $q\ll 1.$

Let us also put  $H_{1,1}(s_2,\si):=H_1(0,s_2,\si).$ Note that $|H_{1,1}|\sim 1, $ since $|H_1|\sim 1$ (cf. \eqref{J12}).  Expanding in powers of $q,$ we  then find that
\begin{eqnarray}\nonumber
\Phi_2(z, s_2, y,\si)&=&q^3z^{\frac 32}H_{1,1}(s_2,\si)+q^2z(s_2G_1(s_2, \sigma)-y_1)\\
&&\hskip-1cm + 
s_2^2G_5(s_2, \sigma)-s_2^{m_1}G_3(s_2, \sigma)y_1 
 -s_2y_2-y_3+O(q^4).\label{Phi23}
\end{eqnarray}

\smallskip

In order to get more precise information on the critical point $s_2^c,$ let us again decompose  
$$
G_1(s_2,\si)=G_{1,1}(\si)+s_2 G_{1,2}(s_2, \si),
$$
as in \eqref{G1decomp}, and similarly
$$
H_{1,1}(s_2,\si)=H_{1,2}(\si)+s_2H_{1,3}(\si) +s_2^2 H_{1,4}(s_2,\si).
$$
Then we may re-write
\begin{eqnarray}\nonumber
\Phi_2(z, s_2, y,\si)&=&q^3 z^{\frac 32}H_{1,2}(\si)-q^2 zy_1 -y_3
-s_2\big(y_2-q^2 z G_{1,1}(\si)-q^3z^{\frac 32}H_{1,3}(\si) \big)\\
&& \hskip3cm +s_2^2 G_6(q^2z,q^3z^{\frac 32}, s_2,y_1,\si)+O(q^4),
\label{Phi24}
\end{eqnarray}
where 
$$
G_6(q^2z,q^3z^{\frac 32}, s_2,y_1,\si):= G_5(s_2,\si) +q^3z^{\frac 32} H_{1,4} (s_2,\si)-s_2^{m_1-2} G_3(s_2,\si) y_1 
+q^2 z G_{1,2} (s_2,\si).
$$
Notice that   $|G_6|\sim 1,$  since $|G_5|\sim 1.$ We thus see that the critical point is of the form 
$$
s_2^c=\Big(y_2-q^2 z G_{1,1}(\si) -q^3z^{\frac 32}H_{1,3}(\si)+O(q^4)\Big) H_{2}(q^2z,q^3z^{\frac 32},y_1,s_2,\si),
$$
where $H_2$ is smooth  and can be assumed to be very close to $1/(2G_6),$ so that in particular it is of size $|H_2|\sim 1.$ It will become important later to observe  that, since $|s_2^c|\ll 1, q\ll1,$ then also necessarily $|y_2|\ll 1.$ 

 Plugging this into \eqref{Phi24}, we find that  after applying the method of stationary phase to the integration in $s_2,$ the new phase is of the form
\begin{eqnarray*}
\Phi_3(z, y,\si)&:=&\Phi_2(z, s_2^c, y,\si)=q^3 z^{\frac 32}H_{1,2}(\si)-q^2 zy_1 -y_3
\\
&&+\Big(y_2-q^2 z G_{1,1}(\si) -q^3z^{\frac 32}H_{1,3}(\si)\Big)^2 H_3(q^2z,q^3z^{\frac 32},y_1,\si)+O(q^4),
\end{eqnarray*}
where  $H_3=H_2(H_2G_6-1)$ is smooth and  very close to $-H_2/2,$ so that  $|H_3|\sim 1.$ 
Decomposing also 
$$
H_3(q^2z,q^3z^{\frac 32},y_1,\si)=H_{3,1}(y_1,\si) +q^2zH_{3,2}(y_1,\si)+q^3z^{\frac 32}H_{3,3}(y_1,\si)+O(q^4),
$$
we see that 
\begin{eqnarray*}\nonumber
\Phi_3(z, y,\si)&=&-(y_3-y_2^2 H_{3,1}(y_1,\si)) \\
&& \hskip-2cm+q^2z\Big(-y_1+y_2^2H_{3,2}(y_1,\si) -2y_2 G_{1,1}(\si) H_{3,1}(y_1,\si) \Big)+q^3 z^{\frac 32}H_{4}(y_1,y_2,\si)+O(q^4),  \end{eqnarray*}
where
$$
H_{4}(y_1,y_2,\si):=H_{1,2}(\si) -2y_2 H_{1,3}(\si)  H_{3,1}(y_1,\si) +y_2^2 H_{3,3}(y_1,\si).
$$
Since $|y_2|\ll 1,$ we see that also $|H_4|\sim 1.$ 
Recall also that  $|G_{1,1}|\ll  1$ is small. By the implicit function theorem, we may therefore write 
\begin{eqnarray*}
-y_1+y_2^2H_{3,2}(y_1,\si) -2y_2G_{1,1}(\si) H_{3,1}(y_1,\si)
 = (y_1-\vp(y_2, \si))H_5(y_1, y_2, \si),
\end{eqnarray*}
with $|H_5|\sim 1.$ We also put $\psi(y_1,y_2,\si):=y_2^2 H_{3,1}(y_1,\si).$  Then  the phase takes on the form
\begin{eqnarray}\nonumber
\Phi_3(z, y,\si)&=&-(y_3-\psi(y_1,y_2,\si))+q^2z(y_1-\vp(y_2, \si))H_5(y_1, y_2, \si) \\
&& \hskip1cm+q^3 z^{\frac 32}H_{4}(y_1,y_2,\si) +O(q^4),   \label{Phi32}
\end{eqnarray}
where $|H_4|\sim 1$ and $|H_5|\sim 1,$ and
 \begin{eqnarray}\nonumber
&&\mu_{k,1}^\la(y+\Ga(\si))
 =\la  2^{\frac k2}
 \int e^{-i \la s_3\Phi_3(z, y,\si)}
\chi_1(z) \,  \chi_0(s_3s')\chi_1(s_3)\\
&&\hskip8cm a\big(z,q, s_2,\si\big)\,ds_3 d z,\label{muAk4}
\end{eqnarray}
where we recall that $q=(2^{k}\la^{-1})^{\frac 13}.$

Choose again a small positive number $\de>0.$ If $q^2|y_1-\varphi(y_2, \si)|>q^{3-\de},$  then we can repeatedly  integrate by parts in $z$ to see that  $\mu_{k,1}^\la(y+\Ga(\si))=O((\la2^k)^{-N})$ for every $N,$ since $\la q^{3-\de}=2^{k(1-\de/3)}\la^{\de/3}.$
Similarly,  if $q^2|y_1-\varphi(y_2, \si)|\le q^{3-\de}$ and  $|y_3-\psi(y_1,y_2,\si)|>q^{3-2\de},$ then the term  
$-(y_3-\psi(y_1,y_2,\si))$
becomes dominant in $\Phi_3,$  so that we can use  integrations by parts in $s_3$  to arrive at the same kind of estimate.

 Thus, the contributions of these regions to the maximal operator are well under control, and we  are left with the control of the contribution by the region where  
$|y_3-\psi(y_1,y_2,\si)|\leq q^{3-2\de}$ and
$|y_1-\varphi(y_2, \si)|\le q^{1-\de}.$

 To this end, observe that $|\pa_z^2\Phi_3(z, y,\si)|\gtrsim q^3.$ We may thus apply van der Corput's lemma to see that in this region, we have that  
$$|\mu_{k,1}^\la(y+\Ga(\si))|\lesssim \la  2^{\frac k2}  2^{-\frac {k}2}=\la
$$ 
(note here that $\la q^3=2^k$).

Hence, by Proposition \ref{maxproj} we obtain similarly as before  that for every $\de>0,$ 
\begin{equation}\label{maxlonepr}
\|\M_{k,1}^\la \|_{L^{1+\ve}\mapsto L^{1+\ve}}\le C_\ve T
\la^{\frac23+\de} 2^{k(\frac13-\de)},
\end{equation}
uniformly in $k$ and $\si.$  Combining the estimates \eqref{l2estl} and \eqref{maxlpr} we find that 
$$
\|\M_{k}^\la\|_{L^2\mapsto L^2}\lesssim T^{\frac 12}2^{-\frac k6} \la^{-\frac 13},
$$
and similar from \eqref{maxkinftyone} and \eqref{maxlonepr} we obtain 
that for every  $\ve >0$ and $\de>0$  sufficiently small, 
$$
\|\M_{k}^\la \|_{L^{1+\ve}\mapsto L^{1+\ve}}\le C_\ve 
T\la^{\frac23+\de} 2^{k(\frac13-\de)},
$$

Interpolating these estimates leads to the estimate \eqref{maxkest1}, which concludes the proof of Lemma \ref{Mlaest}, hence also that of Theorem \ref{nondegairypar}.

\setcounter{equation}{0}
\section{Estimation of  the maximal operator $\M$ when surface has  a\\ $D_4^+$  - type singularity}\label{D4est}

We assume now that $\phi$ has a singularity of type $D_4^+,$ so that we may  assume that  its principal part has the form 
$$
\phi_\pr\x= x_1x_2^2+ x_1^3
$$
(compare Remark \ref{D4}). The  associated principal weight is given by $\ka=(1/3,1/3),$ and the corresponding dilations are given by $\de_r(x_1,x_2)=(r^{\frac 13} x_1, r^{\frac 13}x_2), \ r>0.$

Following our discussion in Section \ref{adaptedc}, by means of a dyadic decomposition, we may  reduce ourselves to the estimation of the maximal operators $\M_k$ , given by  

$$
\M_kf(y,y_3):=\sup_{t>0}|f*(\mu_k)_t(y,y_3)|,
$$
where the re-scaled  measures $\mu_k$ are defined by 
$$
\int f\, d\mu_k:=\int f(x,2^k+\phi^k(x))\, \eta(\delta_{2^{-k}}x)\chi_1(x) dx.
$$
Here we have set 
$$\phi^k(x):=2^k\phi(\de_{2^{-k}}(x))=\phi_\pr(x)+2^k\phi_r(\de_{2^{-k}}(x)).$$ 
Recall that the perturbation  term $2^k\phi_r(\de_{2^{-k}}(\cdot))$ is of order
$O(2^{-\ve k})$ for some $\ve >0.$
By inequality \eqref{mtomk}, we then have 
\begin{equation}\label{mtomk2n}
\|\M\|_{L^p\to L^p}\le \sum_{k= k_0}^\infty  2^{-\frac{2k}3}\|\M_k\|_{L^p\to L^p},
\end{equation}
where we may assume that $k_0$ is sufficiently large. We shall prove that if  $p>h=1/|\ka|=3/2,$ then for every $\de>0$ and $k$ sufficiently large $k,$
\begin{equation}\label{Mkest3}
\|\M_kf\|_p\le C_\de 2^{k(\frac{1}{p}+\de)}\|f\|_p.
\end{equation}
In combination with \eqref{mtomk2n}, this will imply that the maximal operator $\M$ is indeed $L^p$-bounded for every $p>3/2.$ 

\medskip
To this end, recall that  the Hessian determinant of $\phi_\pr$ is given by 
$\Hess(\phi_\pr)(x_1,x_2)=12x_1^2-4x_2^2.$ In order to prove \eqref{Mkest3},  it is sufficient to show that, given any point $x^0$ in the support of $\chi_1$ (which is contained in an annulus on which $|x|\sim 1$), there exist a smooth bump function $\al\ge 0$ supported on a sufficiently small neighborhood of $x^0$ with  $\al(x^0)=1$ so that the corresponding localized maximal operator $\M_k^\al$ satisfies the estimate \eqref{mtomk2}.

If $\Hess(\phi_\pr)(x^0)\ne 0, $ then this has already been shown in Section \ref{adaptedc} (compare estimate \eqref{jest1}). We shall therefore assume that $\Hess(\phi_\pr)(x^0)=0,$
i.e.,  
$$
12(x_1^0)^2-4(x_2^0)^2=0.
$$
But, since $|x^0|\sim 1,$ this implies  
$x_1^0\neq0$ and  $x_2^0\neq0$. By homogeneity, we may then even assume that, say, $x_1^0=1,$ so that $x_2^0=\sqrt{3}$. Then first we shift
 the point $x^0=(1,\,\sqrt{3})$ to the origin in $\bR^2$ by changing to the linear coordinates $z=(z_1,z_2)$ defined by 
 $x_1=1+z_1,\, x_2=\sqrt{3}+z_2$.
In these coordinates, the principal part of $\phi$ is given by
 $$
\phi_\pr(1+z_1,\,
\sqrt{3}+z_2)=4+6z_1+2\sqrt{3}z_2+(\sqrt{3}z_1+z_2)^2+z_1^3+z_1z_2^2.
 $$

Thus, by a linear change of variables in  the ambient space the principal part can be reduced
to the form
$$
\widetilde{\phi_\pr}(y_1,y_2):=\phi_\pr(y_1,y_2-\sqrt{3} y_1)=y_2^2+y_1^3+y_1(y_2-\sqrt{3}y_1)^2=y_2^2+4y_1^3+y_1y_2^2-2\sqrt{3}y_1^2y _2.
$$
Note that the principal part of the function $\widetilde{\phi_\pr}$ is
$y_2^2+4y_1^3$.  Thus, in these coordinates $(y_1,y_2)$  near $(0,0)$ (which corresponds to our original   point $x^0$),    $\tilde\phi$ has a singularity of type $A_2$ and depends smoothly on the small parameter $\si:=2^{-k/3},$ and consequently the estimate \eqref{Mkest3} for $\M_k^\al$ in place of $\M_k$ follows   from Theorem \ref{nondegairypar}.
\smallskip

This completes the proof of Theorem \ref{s1.3}.

\end{document}